\documentclass{amsart}
\usepackage{amscd,amsfonts,amssymb}
\usepackage{textgreek}
\usepackage{tikz}
\usepackage{pgfplots}
\usetikzlibrary{automata,positioning,calc,trees}
\usetikzlibrary{intersections,pgfplots.fillbetween}
\usepackage{graphicx,tikz}
\usepackage{pgf,tikz,pgfplots}
\usepackage{mathrsfs}
\usetikzlibrary{arrows}
\usepackage{enumerate}
\usepackage[shortlabels]{enumitem}
\usepackage{mathrsfs}
\usepackage{amssymb,amsmath,amsthm,color}
\usepackage{caption,subcaption}
\usepackage{hyperref}
\usepackage{cleveref}
\usepackage[nobysame,bibtex-style]{amsrefs}
\usepackage{url}
\usepackage{setspace}
\usepackage{float}
\usepackage{makecell}
\tikzstyle{startstop} = [rectangle, rounded corners, minimum width=3cm, minimum height=1cm,text centered, draw=black]
\tikzstyle{process} = [rectangle, rounded corners, minimum width=3cm, minimum height=1cm, text centered, draw=black]
\tikzstyle{arrow} = [thick,->,>=stealth]

\BibSpec{article}{%
	+{}  {\PrintAuthors}                {author}
	+{,} { \textit}                     {title}
	+{.} { }                            {part}
	+{:} { \textit}                     {subtitle}
	+{,} { \PrintContributions}         {contribution}
	+{.} { \PrintPartials}              {partial}
	+{,} { }                            {journal}
	+{}  { \textbf}                     {volume}
	+{}  { \PrintDatePV}                {date}
	+{,} { \issuetext}                  {number}
	+{,} { \eprintpages}                {pages}
	+{,} { }                            {status}
	+{,} { \url}                        {url}    % <---- ADDED
	+{,} { \PrintDOI}                   {doi}
	+{,} { available at \eprint}        {eprint}
	+{}  { \parenthesize}               {language}
	+{}  { \PrintTranslation}           {translation}
	+{;} { \PrintReprint}               {reprint}
	+{.} { }                            {note}
	+{.} {}                             {transition}
}

\date{\today}
\allowdisplaybreaks
\newcommand{\R}{{\mathbb R}}       % Field of real numbers

\newcommand{\Z}{{\mathbb Z}}       % Ring of integer numbers

\newtheorem{theorem}{Theorem}[section]
\newtheorem{lemma}[theorem]{Lemma}

\newtheorem{coro}[theorem]{Corollary}
\newtheorem{proposition}[theorem]{Proposition}

\newtheorem*{theorem*}{Theorem}

\theoremstyle{definition}
\newtheorem{definition}[theorem]{Definition}

\numberwithin{equation}{section}

\begin{document}
	
	\title[Singular integrals of Christ-Journ\'e type]{On pointwise convergence of singular integrals of Christ-Journ\'e type}
	\author[Bhojak]{Ankit Bhojak}
	\address{Ankit Bhojak\\
		Department of Mathematics\\
		Indian Institute of Science Education and Research Bhopal\\
		Bhopal-462066, India.}
	\email{ankitb@iiserb.ac.in,ankitbhojak101@gmail.com}

	\author[Kundu]{Sangita Kundu}
	\address{Sangita Kundu\\
		Department of Mathematics\\
		Indian Institute of Science Education and Research Bhopal\\
		Bhopal-462066, India.}
	\email{sangita24@iiserb.ac.in}

	\author[Shrivastava]{Saurabh Shrivastava}
	\address{Saurabh Shrivastava\\
		Department of Mathematics\\
		Indian Institute of Science Education and Research Bhopal\\
		Bhopal-462066, India.}
	\email{saurabhk@iiserb.ac.in}

	\begin{abstract}
	The primary objective of this paper is to establish the pointwise almost everywhere convergence of the singular integrals of Christ-Journ\'e type. To this end, we prove the weak type $(1,1)$ estimate for the corresponding maximal operator defined as
	\[
	T^*[a] f(x):= \sup_{r>0}\left|\int_{|x-y| > r} K(x-y)  \int_0^1 a(tx + (1-t)y)\, dt f(y)\, dy\right|,
	\]
	where $K$ is a Calder\'on--Zygmund kernel and $a\in L^\infty$. This result resolves a problem posed by Seeger (Rev. Mat. Iberoam., 2014).
	\end{abstract}

	\maketitle
	\section{Introduction and Main Result}	
	Let $K:\mathbb{R}^d \setminus \{0\}\to\R$ be a  Calder\'on-Zygmund kernel, i.e., there exist  constants $A>0$ and $\varepsilon \in (0,1]$ such that
	\begin{align}
		|K(x)| &\leq  \frac{A}{|x|^{d}}, \quad x \neq 0,
		\label{1.1}\\
		|K(x+h) - K(x)| &\leq A \frac{|h|^\varepsilon}{|x|^{d+\varepsilon}}, \quad \text{for } |x| > 2|h|,\label{1.2}\\
		\|\widehat K\|_{L^\infty(\mathbb{R}^d)} &\leq A.\label{1.3}
	\end{align}
	Note that the conditions \eqref{1.1} and \eqref{1.2} are the standard size and regularity conditions,  respectively. The condition \eqref{1.3} is a cancellation property of the kernel $K$, which guarantees that the principal valued convolution operator $f\mapsto K*f$ is bounded on $L^2(\R^d)$.
    
	For $a \in L^\infty(\mathbb{R}^d),$ let $m_a(x,y)$ denote the average of $a$ over the line segment joining $x$ to $y$; namely 
    $$m_a(x,y):=\int_0^1 a(tx + (1-t)y)\, dt.$$ 
    Let $T[a]$ denote the first--order $d$-commutator associated with the kernel $K$ and the symbol $a$. For $f\in C_c^\infty(\R^d)$, it is defined by
	\begin{equation*}
	T[a]f(x) := p.v. \int K(x-y)m_a(x,y)  f(y)\, dy,\quad a.e.\;x\in\mathbb{R}^d.
	\end{equation*}
    The existence of the above principal value integral for almost every $x\in\R^d$ is equivalent to the following:
    \begin{equation}\label{eqn:pointae}
        \lim_{r\to0^+}\int_{r<|x-y|<1}K(x-y)m_a(x,y)dy\; \text{exists}\;a.e.\;x\in\R^d.
    \end{equation}
    We refer to \cite[Proposition 5.12]{Duoandikoetxea} for details.
	%{\color{red} it would be better to say in $L^2$ sense} where the integral makes sense for an appropriate class of functions $f\in C_c^\infty(\R^d)$. 
    
    We note that for $d=1$, the operator reduces to the classical Calder\'on commutator, see \cite{Cal1965,CMM1982} for precise details. Indeed, if $A$ is a Lipschitz function on $\R$ with $A'=a$, then 
$$m_a(x,y)=\frac{A(x)-A(y)}{x-y}.$$
Consequently, for the kernel $K(x)=1/x$, the kernel of the commutator operator $T[a]$ takes the form $\frac{A(x)-A(y)}{(x-y)^2}$, which is the kernel of the first Calder\'on commutator. These operators were introduced by Calder\'on \cite{Cal1965} in the study of the Cauchy integrals along Lipschitz curves,  which culminated in the work of Coifman, McIntosh and Meyer \cite{CMM1982}. We also refer to  \cite{Cal1975,Muscalu2014,DL2019,Perez1995Commutator, HT2023} for the study of Calder\'on commutators in higher dimensions. 

For $d\geq2$, we note that the operator $T[a]$ does not fall within the theory of Calder\'on commutators.  Christ and Journ\'e \cite{CJ1987} established that $T[a]$ is bounded on $L^p(\R^d)$ for $1<p<\infty$. In fact, they obtained a much stronger result proving polynomial growth estimate in the degree of multilinearity for higher order $d$-commutators defined by 
$$T[a_1,\dots,a_n]f(x) := p.v. \int K(x-y)\prod_{i=1}^n m_{a_i}(x,y)  f(y) dy.$$
We refer the reader to Seeger, Smart and Street~\cite{MR3897981} for a systematic treatment of the higher order commutators with quantitative polynomial bounds in the degree of multilinearity for a general range of H\"older exponents. 

In this article, we are primarily concerned with the end-point estimates for the maximal operator associated with the truncations of the first-order $d$-commutator. We first observe that the kernel of the operator $T[a]$ is given by $K(x-y)m_a(x,y)$ and it satisfies the relevant size condition, 
$$|K(x-y)\,m_a(x,y)|\leq A\,\|a\|_{L^\infty}\,|x-y|^{-d}.$$
However, for a general $a\in L^\infty(\R^d)$ and $d\geq 2$, it need not satisfy any regularity condition. Therefore, the $L^p$ boundedness of $T[a]$ for $1<p<\infty$ does not yield the weak type $(1,1)$ estimate via the classical Calder\'on--Zygmund theory. This is in sharp contrast with the case $d=1$, where $m_a(x,y)=(A(x)-A(y))/(x-y)$ is a Lipschitz function away from the diagonal with $|\partial_y m_a(x,y)|\lesssim \|a\|_{L^\infty}|x-y|^{-1}$, consequently, the standard Calder\'on--Zygmund theory yields the end-point estimates. 

Next, we review some important developments in the context of end-point estimates for the operator $T[a]$. Grafakos and Honz\'ik \cite{GH2012} relied on the $TT^*$ argument along with a certain H\"older-type smoothness of the kernel, specific only to dimension two, to obtain the weak type $(1,1)$ estimate for $T[a]$ in dimension two. Their method is based on the techniques developed by Christ and Rubio de Francia \cite{MR943929} %{\color{red}correct reference} 
and by Hofmann \cite{Hofmann1988} in the context of weak type $(1,1)$ estimate for singular integral operators with rough kernels. The end-point estimate for all dimensions $d\geq2$ was resolved by Seeger \cite{Se2014} by employing a microlocal decomposition of the kernel with non-smooth factor $m_a(x,y)$. Such a decomposition shares an analogy with the microlocal analysis of singular integrals with rough kernels in \cite{Seeger1996}. 
	
Consider the family of truncations $\{T_r[a]:\;r>0\}$ of the $d$-commutator given by
\[T_r[a] f(x):= \int_{|x-y| > r} K(x-y)  m_a(x,y) f(y)\, dy. \]
	For each fixed $r>0$ and $f\in L^1(\R^d)$, the end-point result in \cite{Se2014} implies the uniform estimate 
	\[\sup_{r>0}\|T_r[a]\|_{L^1(\R^d)\to L^{1,\infty}(\R^d)}\lesssim1.\]
	However, the problem concerning the existence of the limit $\lim_{r\to0} T_r[a] f(x)$ for almost every $x\in\R^d$ and $f\in L^1(\R^d)$ is more delicate. It was posed as an open problem by Seeger~\cite{Se2014}. By a well-known principle (refer to Stein \cite{Stein1961} or \cite[Theorem 2.2]{Duoandikoetxea}), this problem reduces to establishing the weak type $(1,1)$ estimate for the corresponding maximal operator $T^*[a]$ defined as
	\[T^*[a] f(x) :=\sup_{r>0} |T_r[a] f(x)|. \]
    We remark here that for the classical Calder\'on--Zygmund operators, the boundedness of the maximal truncated operator is deduced from that of the operator itself by means of Cotlar's inequality. The Cotlar's inequality yields the pointwise almost everywhere control of the maximal function $T^*f$ by the Hardy--Littlewood maximal functions of $f$ and a term involving the operator $T$. The proof of Cotlar's inequality relies essentially on the regularity condition of the kernel, which as explained above, is not available for the operator under consideration when $d\geq2$. Therefore, the weak type $(1,1)$ bound for $T^*[a]$ requires a more sophisticated analysis. It is then natural to explore the microlocal analysis developed by Seeger~\cite{Se2014} in the context of the operator $T^*[a]$. However, we would like to emphasize here that due to the presence of multiple scales simultaneously in the definition of the maximal operator $T^*[a]$, the methods of \cite{Se2014} do not apply directly, and significant new techniques are required to address the problem. We will provide a sketch of the methodology of the proof in the subsequent section. 
	\subsection{Main result:} The main result of this paper is the weak type boundedness of $T^*[a]$ at the end-point $p=1$. This answers the question in \cite{Se2014} affirmatively. 
	\begin{theorem}\label{Maintheorem}
		Let $d\geq2$. Suppose that $K$ satisfies \eqref{1.1}-\eqref{1.3} and $a \in L^\infty(\mathbb{R}^d)$. Then, there exists $C=C(d,\epsilon)>0$ such that for all $\alpha>0$, the following is true.
        \begin{align}\label{thm:weaktype}
			\left| \{ x \in \mathbb{R}^d : T^*[a] f(x) > \alpha \} \right|
			&\leq \frac{CA}{\alpha}\|a\|_{L^\infty(\R^d)}  \|f\|_{L^1(\mathbb{R}^d)}.
		\end{align}
	\end{theorem}
    As an immediate consequence of \Cref{Maintheorem}, we have the  following pointwise covergence result for the singular integral operator $T[a]$.
    \begin{coro}
        Let $d\geq2$. Suppose that $K$ satisfies \eqref{1.1}-\eqref{eqn:pointae} and $a \in L^\infty(\mathbb{R}^d)$. Then, for $f\in L^1(\R^d)$, there exists a set $E_f\subset\R^d$ of measure zero such that the limit
        \[\lim_{r\to0^+}T_r[a]f(x),\]
        exists for every $x\in\R^d\setminus E_f$.
    \end{coro}
    The proof of \Cref{Maintheorem} requires several ingredients that warrant a comprehensive outline of the methodology. Therefore, we first provide an outline of the methodology below.
	\subsection{Methodology of proof of \Cref{Maintheorem}}
		The first step is to establish an $L^2$-estimate of the form:
        \begin{align}\label{thm:strong22}
			\|T^*[a]f\|_{L^2(\R^d)}&\leq CA\|a\|_{L^\infty(\R^d)}\|f\|_{L^2(\R^d)}.
		\end{align}
        This requires a Cotlar-type inequality involving a frequency localization of a dyadic version of the truncations $T_r[a]$ in the spirit of \cite[Theorem E]{DR1986}. We then perform a dyadic regularization of the kernel $K$ into certain Besov spaces (Definition \ref{def:Besov}) as in \cite{MR3897981}, followed by an appropriate Littlewood-Paley decomposition of the non-regular symbol $a$ and the input function into low and high components. After a randomization argument via Khintchine's inequality, this reduces to studying certain paraproducts whose estimates are available from the machinery developed in \cite{MR3897981}. The details of  the proof of \eqref{thm:strong22} are given in \Cref{sec:proofL2}.
		
		We now outline the steps involved in proving the weak type estimate \eqref{thm:weaktype}. 
		\begin{enumerate}
			\item The Calder\'on-Zygmund decomposition of the input function and the $L^2$-boundedness of $T^*[a]$ reduce the problem to estimating the bad part supported on a maximal collection of cubes. 
			\item We employ a partition of unity to obtain a multiscale decomposition of the operator. To estimate the bad part, we will utilize a further partition of the bad cubes based on the shifted dyadic grids (Lemma~\ref{lemma:shifted}) as in \cite{bhojak2026endpoint}.
			\item For a fixed grid, we perform an algorithm based on ideas from \cite{KL2018,KL2020,bhojak2026endpoint} to control the interaction of scales of the operator acting on the bad cubes. This enables us to employ the following version of the Rademacher-Menshov theorem to control the supremum appearing over the scales.
		\begin{lemma}[Rademacher-Menshov Theorem]\label{thm:RM}
		Let $(X,\mu)$ be a measure space and $\{\phi_j:1\le j\le N\}$ a sequence of functions such that  for all sequences of coefficients $c_j\in\{0,\pm1\}$,
		\begin{equation*}
		\left\|
		\sum_{j=1}^{N} c_j\phi_j
		\right\|_{L^2(X)}
		\le A.
		\end{equation*}

		Then
		\begin{equation*}
		\left\|
		\sup_{1\le n\le N}
		\left|
		\sum_{j=1}^{n}\phi_j
		\right|
		\right\|_{L^2(X)}
		\lesssim A\log(2+N).
		\end{equation*}
		\end{lemma}
	We remark that Lemma~\ref{thm:RM} has been known to be crucial to address various maximally truncated operators, for instance, the maximal oscillatory singular integrals in \cite{KL2018,KL2020,IKM2025} and the maximal rough singular integrals in \cite{Lai2025}.
	\item \label{step4} We employ the microlocal decomposition of the kernel from \cite{Se2014} to exploit orthogonality in the $L^2$- estimate (Lemma~\ref{lem 3.4}) for certain pieces of the kernel via a $TT^*$ argument. The error terms are then estimated with a quantitative polynomial growth in the $L^1$-norm (Lemmata \ref{lem 3.1}, \ref{lem 3.2} and \ref{lem 3.3}) in the spirit of Calder\'on-Zygmund theory.
	\item The estimates obtained in step \eqref{step4} are then unified into a single $L^2$ estimate (Proposition \ref{mainlemma}), which enables us to employ Lemma~\ref{thm:RM}. Such an estimate is based on an interpolation scheme with appropriate $L^1$ and $BMO$ estimates for the corresponding multiscale operator.
		\end{enumerate}
    \Cref{sec:proofmain} is devoted to the proof of \Cref{Maintheorem}.
	%\subsection{Outline of the paper} The proofs of \Cref{Maintheorem}\eqref{thm:strong22} and \eqref{thm:weaktype} are carried out in Sections \ref{sec:proofL2} and \ref{sec:proofmain} respectively.

	\section{$L^2$-boundedness of $T^*[a]$: Proof of \eqref{thm:strong22}}\label{sec:proofL2}
Without loss of generality we may assume that $A=1$ and $\|a\|_{L^{\infty}(\mathbb{R}^d)}=1$. Also, by a standard density argument, it is enough to work with $f \in L^1(\mathbb R^d)\cap L^2(\mathbb R^d)$. Let $\psi$ be a non-negative radial $C^\infty(\mathbb{R}^d)$ function such that $\psi$ is supported in
	$
	\left\{ y \in \mathbb{R}^d : \frac{1}{2} \leq |y| \leq 2 \right\}
	$ and
	$
	\sum_{j \in \mathbb{Z}} \psi(2^{-j} y) = 1$ for all $y \neq 0$. Then, we have
	\[
	K(y) = \sum_{j \in \mathbb{Z}} K(y)\psi(2^{-j}y),\quad y \neq 0. 
	\]
	
	We set 
	$
	K_j(y) = K(y)\psi(2^{-j}y). 
	$
	For ease of notation, we define $T_j$ to be the integral operator given by
	\[
	T_jf(x)=\int K_j(x-y)m_a(x,y) f(y)dy. 
	\]
	Indeed, for $2^{-N}\leq r<2^{-N+1}$ the kernels of $T_r[a]$ and $\sum_{j>-N}T_j$ differ only where $|x-y|\approx 2^{-N}$, where both are $O(2^{Nd})$. Therefore, a standard discretization of the supremum implies that
	\begin{align}\label{discritazation}
    T^*[a] f(x)\lesssim Mf(x)+\sup_{N \in \mathbb{Z}} \left| \sum_{j > -N} T_jf(x) \right|,
    \end{align}
	where $Mf$ is the Hardy-Littlewood maximal function.

	We set $\Psi_N(y)=\sum_{j>-N}\psi(2^{-j}y)$ and define $\widetilde{T}_N$ as follows.
	\[\widetilde{T}_N[a]f(x):=\int_{\mathbb{R}^d}K(x-y)m_a(x,y) \Psi_N(x-y)f(y)\ dy.\]
	Let $\varPhi$ be a radial non-negative $C^\infty(\mathbb{R}^d)$ function supported in $\{x \in \mathbb{R}^d:|x|\leq \frac{1}{4}\}$ and satisfying $\int \varPhi(x)\,dx = 1. $ 
	Let $\varPhi^{(N)}(x) := 2^{Nd}\varPhi(2^{N}x)$ and $P^{(N)}$ denote the integral operators defined as
	$P^{(N)}f(x):=\varPhi^{(N)}*f(x)$. Set $\Delta_N:=P^{(N)}-P^{(N-1)}$. 
    Since $\Psi_N=\sum_{j>-N}\psi(2^{-j}\cdot)$, we have $\widetilde T_N[a]=\sum_{j>-N}T_j$. Writing $\widetilde T_N[a] = P^{(N)}T[a] - P^{(N)}\sum_{j\le -N}T_j + (I-P^{(N)})\widetilde T_N[a]$, we obtain
	\begin{align}\sup_{N \in \mathbb{Z}} \left| \sum_{j > -N} T_jf(x) \right|\leq \sup_{N\in\Z} \left|P^{(N)}T[a]f(x)\right|+\sup_{N\in\Z}\left|P^{(N)}\left(\sum_{j \leq -N} T_jf \right)(x)\right|+\sup_{N\in\Z}\left|(I-P^{(N)})\widetilde{T}_Nf(x)\right|.\label{2.2}\end{align}

	The first two terms in the expression above are easy to estimate. Indeed, the $L^2$-bounds follow from the $L^2$-boundedness of $M$ and $T[a]$ along with the inequalities below.
	\begin{align*}
		\sup_{N\in\Z} \left|P^{(N)}T[a]f(x)\right|&\leq MT[a]f(x),\\
		\sup_{N\in\Z}\left|P^{(N)}\left(\sum_{j \leq -N} T_jf \right)(x)\right|&\lesssim M_qf(x) +M(Mf)(x),\ \text{for some}\ 1<q<2,
	\end{align*}
    where the last inequality follows from the following pointwise 
	estimate and the $L^q$-boundedness of $T[a]$.
    \small
    \begin{align*}
    \left|P^{(N)}\left(\sum_{j \leq -N} T_jf \right)(x)\right|&\le |P^{(N)}T[a]f_N(x)|+\\
    &\quad\int_{|z|\leq2^{-N-2}}\varPhi^{(N)}(z)\int_{2^{-N}<|x-z-y|\leq 2^{-N+3}}
|K\Psi_N(x-z-y)||m_a(x-z,y)|\,f(y)\ dy\ dz,
    \end{align*}
    where $f_N:=f\chi_{B(x,2^{-N+2})}$.
    %{\color{red} we should elaborate a little here to convince the estimate above.}%Note that the kernel of the operator $\sum_{j \leq -N} T_j$ is $(1-\Psi_N)K(x-y)m_a(x,y)$ and $\Psi_NK$ satisfies \labelcref{1.1,1.2,1.3}.
	Thus, it remains to estimate the third term in \eqref{2.2}. For any function $J:\R\times \R^d\to\R$ and $\delta>0$, we set $J^{(\delta)}(\alpha,x):=\delta^dJ(\alpha,\delta x)$. For each $j\in\Z$, we define the  operator $W_j$ as follows 
	\[
	W_{j}[J,a](f)(x)
	:=\int_{\mathbb{R}}\int_{\mathbb{R}^d}J^{(2^j)}(\alpha,v)a(x-\alpha v)f(x-v)\,dv\,d\alpha.\]
	We define the kernels $\mu_{j,k}$ and $\zeta_j$ for $j,k\in\Z$ as follows.
\[\mu_{j,k}(\alpha,v):=\zeta_j(\alpha,v)\Psi_0(2^{k}v),\quad \zeta_j^{(2^j)}(\alpha,v):=\Delta_{j}K(v)\chi_{[0,1]}(\alpha).\]
The kernels defined above lie in certain special class of Besov spaces introduced in \cite{MR3897981}.
\begin{definition}\label{def:Besov}
	For $0 \leq \varepsilon \leq 1$, we define four (semi-)norms as follows.
\begin{align*}
& \|\varsigma\|_{\mathcal{B}_{\varepsilon, 1}}:=\iint\left(1+\left|\alpha\right|\right)^{\varepsilon}|\varsigma(\alpha, v)| d \alpha d v.  \\
& \|\varsigma\|_{\mathcal{B}_{\varepsilon, 2}}:=\sup _{\substack{0<h \leq 1
}} h^{-\varepsilon} \iint\left|\varsigma\left(\alpha+h , v\right)-\varsigma(\alpha, v)\right| d \alpha d v,  \\
& \|\varsigma\|_{\mathcal{B}_{\varepsilon, 3}}:=\sup _{0<|h| \leq 1}|h|^{-\varepsilon} \iint|\varsigma(\alpha, v+h)-\varsigma(\alpha, v)| d \alpha d v. \\
& \|\varsigma\|_{\mathcal{B}_{\varepsilon, 4}}:=\iint(1+|v|)^{\varepsilon}|\varsigma(\alpha, v)| d \alpha d v 
\end{align*}
%{\color{red} In the definition above $\alpha\in \R$ so no need to take $e_i$.}
Let $\mathcal{B}_{\varepsilon}\left(\mathbb{R} \times \mathbb{R}^d\right)$ be the space of those $\varsigma \in L^1\left(\mathbb{R} \times \mathbb{R}^d\right)$ such that the following is true

\begin{equation*}
\|\varsigma\|_{\mathcal{B}_{\varepsilon}}:=\|\varsigma\|_{\mathcal{B}_{\varepsilon, 1}}+\|\varsigma\|_{\mathcal{B}_{\varepsilon, 2}}+\|\varsigma\|_{\mathcal{B}_{\varepsilon, 3}}+\|\varsigma\|_{\mathcal{B}_{\varepsilon, 4}}<\infty.
\end{equation*}

\end{definition}
\begin{lemma}\label{lemma:estimate} The following statements hold true.
	\begin{enumerate}
		\item For all $j\in\Z$ and $\alpha\in\mathbb{R}$, we have
		\begin{equation}\label{2.3}\int \zeta_j(\alpha,v)\,dv=0.\end{equation}
		\item For any $0<\delta<\varepsilon$, we have
		\[\sup_{j\in\mathbb{Z}}\|\zeta_j\|_{B_\delta}\lesssim _{\delta,\varepsilon,d} 1.\] 
		\item For all $j,k\in\Z$ and $0<\delta<\varepsilon$, we have
		\begin{equation}\label{2.4}\|\mu_{j,k}\|_{L^1(\mathbb{R}\times\mathbb{R}^d)}\le\|\zeta_{j}\|_{L^1(\mathbb{R}\times\mathbb{R}^d)},\quad\text{and}\quad\|\mu_{j,k}\|_{\mathcal{B}_{\delta}(\mathbb{R}\times\mathbb{R}^d)}\lesssim_{\delta,\varepsilon,d} 1.\end{equation}
	\end{enumerate}
\end{lemma}
  %{\color{red} Do we need $\|\mu_{j,k}\|_{\mathcal{B}_{\delta}(\mathbb{R}\times\mathbb{R}^d)}\le \|\zeta_{j}\|_{\mathcal{B}_{\delta}(\mathbb{R}\times\mathbb{R}^d)}$ or simply $\lesssim 1$ is enough.}
\begin{proof}
	\eqref{2.3} is trivial. For (2) see Example 2.2 and Proposition 3.2. in  \cite{MR3897981}. 

   % {\color{red} Do we need $\|\mu_{j,k}\|_{\mathcal{B}_{\delta}(\mathbb{R}\times\mathbb{R}^d)}\le \|\zeta_{j}\|_{\mathcal{B}_{\delta}(\mathbb{R}\times\mathbb{R}^d)}$ or simply $\lesssim 1$ is enough.}
The inequalities $\|\mu_{j,k}\|_{L^1(\mathbb{R}\times\mathbb{R}^d)}\le\|\zeta_{j}\|_{L^1(\mathbb{R}\times\mathbb{R}^d)}$ and $\|\mu_{j,k}\|_{\mathcal{B}_{\delta,i}(\mathbb{R}\times\mathbb{R}^d)}\le \|\zeta_{j}\|_{\mathcal{B}_{\delta,i}(\mathbb{R}\times\mathbb{R}^d)}$ for $i=1,2,4$ are straightforward. Now for $i=3$, note that if $k< 0$, then we have
	\[\|\mu_{j,k}\|_{\mathcal{B}_{\delta,3}(\mathbb{R}\times\mathbb{R}^d)}\lesssim
	\|\zeta_j\|_{\mathcal{B}_{\delta,3}(\mathbb{R}\times\mathbb{R}^d)}
	+\|\zeta_j\|_{\mathcal{B}_{\delta,1}(\mathbb{R}\times\mathbb{R}^d)}.\]
	For $k \ge 0$, we write 
\begin{align*}
    \|\mu_{j,k}\|_{\mathcal{B}_{\delta,3}(\mathbb{R}\times\mathbb{R}^d)}
	\le
	\|\zeta_j\|_{\mathcal{B}_{\delta,3}(\mathbb{R}\times\mathbb{R}^d)}
	+
	\sup_{0<|h|\leq 1}
	|h|^{-\delta}
	\int_{\mathbb{R}}\int_{|v|\le2^{-k+1}+|h|}
	|\zeta_j(\alpha,v)||\Psi(2^k(v+h))-\Psi(2^k v)|\,d\alpha\,dv.
\end{align*}
Next, we consider the two cases $0<|h|\leq2^{-k}$ and $2^{-k}<|h|\leq1$ separately. 

\subsection*{Case I} For $0<|h|\leq2^{-k}$, we use the mean value theorem, the Cauchy-Schwarz inequality, Plancherel's theorem and (\ref{1.3}), to see that the last term is dominated by
\begin{align*}
|h|^{1-\delta}\,
2^k\|\nabla\Psi\|_{L^\infty(\mathbb{R}^d)}
(2^{-k+1}+|h|)^{d/2}
\left[
\int
\left|
\big(\hat{\varPhi}(\xi)-\hat{\varPhi}(2\xi)\big)
\,\hat{K}(2^j\xi)
\right|^2\,d\xi
\right]^{1/2}\lesssim 2^{-k(\frac{d}{2}-\delta)}\lesssim 1.
\end{align*}
\subsection*{Case II} For $2^{-k}<|h|\leq1$, we apply the Cauchy-Schwarz inequality, Plancherel's theorem and (\ref{1.3}) to obtain that the last term is dominated by 
\begin{align*}
|h|^{-\delta}\,
(2^{-k+1}+|h|)^{d/2}
\left[
\int
\left|
\big(\hat{\varPhi}(\xi)-\hat{\varPhi}(2\xi)\big)
\,\hat{K}(2^j\xi)
\right|^2\,d\xi
\right]^{1/2}\lesssim1.
\end{align*}
\end{proof}

In view of the definition of $W_j$, we write
\begin{align*}
    &\sup_{N\in\Z}\left|(I-P^{(N)})\widetilde{T}_Nf(x)\right|\\
	\leq&\sup_{N \in \mathbb{Z}}|(I-P^{(N)})\sum_{j>N}W_{j}[\mu_{j,N-j},a]f(x)|+\sup_{N \in \mathbb{Z}}|(I-P^{(N)})\sum_{j\leq N}W_{j}[\mu_{j,N-j},a]f(x)|\\
    =:&\mathcal{G}_1^*[a]f(x)+\mathcal{G}_2^*[a]f(x).
\end{align*}
\subsection*{Estimate for $\mathcal{G}_1^*[a]$:} We use the fact $\int\varPhi=1$ and (\ref{1.2}) to obtain, for $|\nu|>2^{-j}$, 

\begin{align*}
    |\Delta_{j}K(v)|&=
\left|
\int_{|y|\leq 2^{-j-1}}
\left[
2^{jd}\varPhi(2^j y)
-
2^{(j-1)d}\varPhi(2^{j-1}y)
\right]
\bigl(K(v-y)-K(v)\bigr)\,dy
\right|\\
& \leq \frac{1}{|v|^{d+\varepsilon}}
\int
|y|^\varepsilon
\left|
2^{jd}\varPhi(2^j y)
-
2^{(j-1)d}\varPhi(2^{j-1}y)
\right|\,dy \\
&\lesssim
\frac{2^{-j\varepsilon}}{|v|^{d+\varepsilon}}.
\end{align*}
This implies that
\begin{align*}
   \big |\sum_{j>N}W_{j}[\mu_{j,N-j},a]f(x)\big|&\lesssim\sum_{j>N}2^{-j\varepsilon}\int_{|v|\ge 2^{-N+1}}\frac{1}{|v|^{d+\varepsilon}}|f(x-v)|dv\\
   &=\sum_{j>N}2^{-j\varepsilon}\sum_{k=0}^\infty\int_{2^k \ 2^{-N+1}\le|v|\le 2^{k+1} \ 2^{-N+1}}\frac{1}{|v|^{d+\varepsilon}}|f(x-v)|dv\\
   &\lesssim\sum_{j>N}2^{-(j-N)\varepsilon}\sum_{k=0}^\infty\frac{1}{2^{k\varepsilon}}Mf(x)\\
   &\lesssim Mf(x).
\end{align*}
Thus $\mathcal{G}_1^*[a]f(x)\lesssim M(Mf)(x)$. Consequently, we get that 
$$\|\mathcal{G}_1^*[a]f\|_2\lesssim\|f\|_2.$$
%%%%%%%%%%%%%%%%
\subsection*{Estimate for $\mathcal{G}_2^*[a]$:} We recall that $I-P^{(j)}=\sum_{r\geq 1}\Delta_{j+r}$, with convergence in the strong operator topology (as operators from $L^p(\mathbb{R}^d) \to L^p(\mathbb{R}^d)$, $1<p<\infty$). Thus, we write
\begin{align*}
(I-P^{(N)})\sum_{j\leq N}W_j[\mu_{j,N-j},a]f
&=\mathcal{H}_{1,N}[a]f+\mathcal{H}_{2,N}[a]f+\mathcal{L}_N[a]f,
\end{align*}
where
\begin{align*}
\mathcal{H}_{1,N}[a]f
&:=(I-P^{(N)})\sum_{k\geq0}
W_{N-k}[\mu_{N-k,k},a](I-P^{(N-k)})f,
\\
\mathcal{H}_{2,N}[a]f
&:=(I-P^{(N)})\sum_{k\geq0}
W_{N-k}[\mu_{N-k,k},(I-P^{(N-k)})a]P^{(N-k)}f,
\end{align*}
and
\begin{align*}
\mathcal{L}_N[a]f
&:=(I-P^{(N)})\sum_{k\geq0}
W_{N-k}[\mu_{N-k,k},P^{(N-k)}a]P^{(N-k)}f.
\end{align*}
Thus we have
\[\mathcal{G}_2^*[a]f(x)\leq\mathcal{H}_{1}^*[a]f(x)+\mathcal{H}_{2}^*[a]f(x)+\mathcal{L}^*[a]f(x),\]
where
\[\mathcal{H}_{l}^*[a]f(x)=\sup_{N\in\Z}|\mathcal{H}_{l,N}[a]f(x)|,\;l=1,2,\quad \mathcal{L}^*[a]f(x)=\sup_{N\in\Z}|\mathcal{L}_N[a]f(x)|.\]
\subsection*{Estimate for $\mathcal{H}_{1}^*[a]$} We write
\begin{align*}
\mathcal{H}_{1,N}[a]f
&=(I-P^{(N)})\sum_{k\geq0}
W_{N-k}[\mu_{N-k,k},a](I-P^{(N-k)})f\\
&=\sum_{k\geq0}\sum_{\ell\geq1}\sum_{r\geq1}
\Delta_{N+\ell}
W_{N-k}[\mu_{N-k,k},a]
\Delta_{N-k+r}f.
\end{align*}

We will require the following result from \cite[Theorem 10.1]{MR3897981}, which is due to an $L^2$ estimate for rough kernels and a standard almost orthogonality lemma involving the Littlewood-Paley operators.
\begin{lemma}\label{Lem8.1}
Let $0<\delta<1$ and suppose that
$\sup_j\|\varsigma_j\|_{B_\delta(\mathbb{R}\times\mathbb{R}^d)}<\infty$ and $b\in L^\infty(\mathbb{R}^d)$. Let $k_1,k_2\in \mathbb{N}.$ Then $\sum_j
\Delta_{j+k_1}W_j[\varsigma_{j},b]\Delta_{j+k_2}$
converges in the strong operator topology, as bounded operators on $L^2(\mathbb R^d)$. Moreover, there is $c>0$ such that, for any $n\in \{0\} \cup  \mathbb{N}$,
\begin{align*}
&\|\sum_{j=-n}^n
\Delta_{j+k_1}W_j[\varsigma_{j},b]\Delta_{j+k_2}\|_{L^2(\mathbb R^d)\to L^2(\mathbb R^d)}\\
&\lesssim
\sup_j\|\varsigma_j\|_{L^1(\mathbb{R}\times\mathbb{R}^d)}
\min\left\{1,2^{-(k_1+k_2)c\delta}\tfrac{\sup_j\|\varsigma_j\|_{\mathcal{B}_\delta(\mathbb{R}\times\mathbb{R}^d)}}{\sup_j\|\varsigma_j\|_{L^1(\mathbb{R}\times\mathbb{R}^d)}}\right\}\|b\|_{L^\infty(\mathbb{R}^d)}
.
\end{align*}
\end{lemma}

By Lemmata~\ref{lemma:estimate} and \ref{Lem8.1}, we have
\begin{align*}
\left\|
\sum_{j=-n}^n\epsilon_j
\Delta_{j+k+l}W_j[\mu_{j,k},a]\Delta_{j+r}
\right\|_{L^2(\mathbb{R}^d)\to L^2(\mathbb{R}^d)}
\lesssim
\,2^{-c\delta(k+l+r)}
\end{align*}
uniformly over signs $\epsilon_j\in \{-1,1\}$. Consequently, by an application of Khintchine’s inequality we get that
\begin{align*}
\left\|
\left(
\sum_j
\left|
\Delta_{j+k+l}W_j[\mu_{j,k},a]\Delta_{j+r}f
\right|^2
\right)^{1/2}
\right\|_{L^2(\mathbb{R}^d)}
\\
&\lesssim
2^{-c\delta(k+l+r)}
\|f\|_{L^2(\mathbb{R}^d)}.
\end{align*}

Therefore, for fixed $k,\ell,r$, we have
\begin{align*}
\left\|
\sup_N
\left|
\Delta_{N+\ell}
W_{N-k}[\mu_{N-k,k},a]
\Delta_{N-k+r}f
\right|
\right\|_{L^2(\mathbb{R}^d)}
&\leq
\left\|
\left(
\sum_N
\left|
\Delta_{N+\ell}
W_{N-k}[\mu_{N-k,k},a]
\Delta_{N-k+r}f
\right|^2
\right)^{1/2}
\right\|_{L^2(\mathbb{R}^d)}\\
&=
\left\|
\left(
\sum_j
\left|
\Delta_{j+k+\ell}
W_j[\mu_{j,k},a]
\Delta_{j+r}f
\right|^2
\right)^{1/2}
\right\|_{L^2(\mathbb{R}^d)}\\
&\lesssim
2^{-c\delta(k+\ell+r)}\|f\|_{L^2(\mathbb{R}^d)}.
\end{align*}

Summing over $k,\ell \ \text{and}\  r$ we get the ${L^2(\mathbb{R}^d)} \to {L^2(\mathbb{R}^d)}$ boundedness of $\mathcal{H}_{1}^*[a]$.
\subsection*{Estimate for $\mathcal{H}_{2}^*[a]$}\label{subsubsec:H_1} We write
\begin{align*}
    \mathcal{H}_{2,N}[a]f
&=(I-P^{(N)})\sum_{k\geq0}
W_{N-k}[\mu_{N-k,k},(I-P^{(N-k)})a]P^{(N-k)}f\\
&=\sum_{k\geq0}\sum_{\ell\geq1}
\Delta_{N+\ell}
W_{N-k}[\mu_{N-k,k},(I-P^{(N-k)})a]
P^{(N-k)}f.
\end{align*}
We will employ the following estimate from \cite[Theorem 11.1]{MR3897981}, where the main ingredients are a special case of the Carleson embedding theorem and a nontrivial square-function estimate.
\begin{lemma}\label{Lem8.2}
Let $0<\delta<1$ and suppose that
$\sup_j\|\varsigma_j\|_{B_\delta(\mathbb{R}\times\mathbb{R}^d)}<\infty$ and $b\in L^\infty(\mathbb{R}^d)$. Let $k_1\in \mathbb{N}.$ Then $\sum_j
\Delta_{j+k_1}W_j[\varsigma_{j},(I-P^{(j)})b]P^{(j)}$
converges in the strong operator topology, as bounded operators on $L^2(\mathbb R^d)$. Moreover, there is $c>0$ such that, for any $n \in \{0\} \cup \mathbb{N}$,
\begin{align*}
&\|\sum_{j=-n}^n
\Delta_{j+k_1}W_j[\varsigma_{j},(I-P^{(j)})b]P^{(j)}\|_{L^2(\mathbb R^d)\to L^2(\mathbb R^d)}\\
&\qquad \lesssim
\sup_j\|\varsigma_j\|_{L^1(\mathbb{R}\times\mathbb{R}^d)}
\min\left\{2^{-k_1c\delta}\tfrac{\sup_j\|\varsigma_j\|_{\mathcal{B}_\delta(\mathbb{R}\times\mathbb{R}^d)}}{\sup_j\|\varsigma_j\|_{L^1(\mathbb{R}\times\mathbb{R}^d)}},\log^{\tfrac{3}{2}}(1+\tfrac{\sup_j\|\varsigma_j\|_{\mathcal{B}_\delta(\mathbb{R}\times\mathbb{R}^d)}}{\sup_j\|\varsigma_j\|_{L^1(\mathbb{R}\times\mathbb{R}^d)}})\right\}\|b\|_{L^\infty(\mathbb{R}^d)}.
\end{align*}
\end{lemma}
We set $A_k:=\sup_{j\in\Z}\|\mu_{j,k}\|_1$ and $B_k:=\sup_{j\in\Z}\|\mu_{j,k}\|_{\mathcal{B}_\delta}$. Thus by Lemmata~\ref{lemma:estimate} and \ref{Lem8.2}, we have 
\begin{align*}
\left\|
\sum_{j=-n}^n\epsilon_j
\Delta_{j+k+l}W_j[\mu_{j,k},(I-P^{(j)})a]P^{(j)}
\right\|_{L^2(\mathbb{R}^d)\to L^2(\mathbb{R}^d)}
&\lesssim
A_k\min\left\{
2^{-c\delta (k+l)}\tfrac{B_k}{A_k},\log^{\tfrac{3}{2}}(1+\tfrac{B_k}{A_k})
\right\}\\
&\lesssim A_k (2^{-c\delta (k+l)}\tfrac{B_k}{A_k})^{\tfrac{1}{4}}\log^{\tfrac{9}{4}}(1+\tfrac{B_k}{A_k})\\
&\lesssim
\,2^{-\tfrac{c}{4}\delta (k+l) } A_k^{\tfrac{1}{2}}B_k^{\tfrac{1}{2}}\\
&\lesssim 2^{-\tfrac{c}{4}\delta (k+l) }.
\end{align*}
Therefore, the ${L^2(\mathbb{R}^d)} \to {L^2(\mathbb{R}^d)}$ boundedness of $\mathcal{H}_{2}^*[a]$ follows similarly as previously. 

%%%%%%%%%%%%%%%%%%%%%
\subsection*{Estimate for $\mathcal{L}^*[a]$} For $j,k\in \mathbb{Z}$, we define 

\[\beta_{j,k}:=\int_{\mathbb{R}}\int_{\mathbb{R}^d}\mu_{j,k}(\alpha,v)d\alpha\,dv\quad\text{and}\quad\eta_{j,k}(\alpha,v):=\frac{\beta_{j,k}}{|B(0,1)|}\chi_{[0,1]}(\alpha)\chi_{B(0,1)}(v).\]
Using the mean-zero property (\ref{2.3}) we get 
\begin{align*}
    |\beta_{N-k,k}|&=\Big|\int_{\mathbb{R}}\int_{\mathbb{R}^d}\zeta_{N-k,k}(\alpha,v)(\Psi(2^k v)-1))d\alpha\,dv\Big|\notag\\
    &=\Big|\int_{\mathbb{R}^d}[(\varPhi^{(2^{N-k})}-\varPhi^{(2^{N-k-1})})*K](v)(\Psi(2^N v)-1))\,dv\Big|\notag\\
    &\lesssim 2^{(N-k)d}\int_{|v|\leq 2^{-N+1}}|\Psi(2^Nv)-1|\,dv\lesssim 2^{-kd}.%\label{2.5}
\end{align*}
First, note that for any $j\in \mathbb{Z}$, we have that $(\varPhi^{(2^{j})}-\varPhi^{(2^{j-1})})*K \in L^1(\mathbb{R}^d)$. This is easy to verify using  Fubini's theorem and the fact that $\zeta_j \in L^1(\mathbb R \times \mathbb R ^d)$. Also, note that the Fourier transform of $(\varPhi^{(2^{j})}-\varPhi^{(2^{j-1})})*K$ lies in $L^1(\mathbb{R}^d).$ Consequently, we have the pointwise estimate 
$$|[(\varPhi^{(2^{j})}-\varPhi^{(2^{j-1})})*K](x)|\lesssim 2^{jd}.$$

We also observe that $\|\eta_{N-k,k}\|_{\mathcal{B}_\delta(\mathbb R \times \mathbb R ^d)}\lesssim_\delta|\beta_{N-k,k}|\lesssim 1.$ Estimates for $\|\eta_{N-k,k}\|_{\mathcal{B}_{\delta,i}(\mathbb R \times \mathbb R ^d)}$ for $i=1,2,4$ are easy to verify. For $i=3$, we consider 
\begin{align*}
\|\eta_{N-k,k}\|_{\mathcal{B}_{\delta,3}(\mathbb R \times \mathbb R ^d)}&=\frac{|\beta_{N-k,k}|}{|B(0,1)|}\sup_{0<|h|\le1}|h|^{-\delta}\int_{\mathbb{R}}\int_{\mathbb{R}^d}|\chi_{B(0,1)-h}(v)-\chi_{B(0,1)}(v)|\chi_{[0,1]}(\alpha)\,d\alpha\,dv\\
&\leq |\beta_{N-k,k}|\sup_{0<|h|\le1}|h|^{-\delta}\int_{1-|h|\leq|v|\leq 1+|h|}dv\\
&\lesssim |\beta_{N-k,k}|.  
\end{align*}
Next, we define $\tau_{N-k,k}:=\mu_{N-k,k}-\eta_{N-k,k}.$ Note that \begin{align}\int_{\mathbb{R}}\int_{\mathbb{R}^d}\tau_{N-k,k}(\alpha,v)\,d\alpha \,dv=0.\label{2.6}\end{align}
Then \[\mathcal{L}_N[a]f=\mathcal{L}_{1,N}[a]f+\mathcal{L}_{2,N}[a]f,\]
 where 
 \begin{align*}
\mathcal{L}_{1,N}[a]f&:=(I-P^{(N)})\sum_{k\geq0}
W_{N-k}[\eta_{N-k,k},P^{(N-k)}a]P^{(N-k)}f\\ 
\mathcal{L}_{2,N}[a]f&:=(I-P^{(N)})\sum_{k\geq0}
W_{N-k}[\tau_{N-k,k},P^{(N-k)}a]P^{(N-k)}f.
 \end{align*}
For the term $\mathcal{L}_{1,N}[a]f$, we have
\begin{align*}
    \Big|\sum_{k\geq0}
W_{N-k}[\eta_{N-k,k},P^{(N-k)}a]P^{(N-k)}f(x) \Big|&\leq \sum_{k \ge 0}\int_{\mathbb{R}}\int_{\mathbb{R}^d}2^{(N-k)d}\frac{|\beta_{N-k,k}|}{|B(0,1)|}\chi_{[0,1]}(\alpha)\chi_{B(0,1)}(2^{N-k}v)\\
&\qquad|P^{(N-k)}a(x-\alpha v)||P^{(N-k)}f(x-v)|\,d\alpha\,dv\\
& \lesssim \sum_{k\ge 0}2^{-kd}2^{(N-k)d}\int_{|v|\le2^{-(N-k)}}Mf(x-v)dv\\
&\lesssim M(Mf)(x).
\end{align*}
Thus \[\sup_{N\in \mathbb{Z}}|\mathcal{L}_{1,N}[a]f(x)|\lesssim M(M(Mf))(x),\] and hence the proof follows for this part.

For the remaining part we write
\begin{align*}
    \mathcal{L}_{2,N}[a]f&=(I-P^{(N)})\sum_{k\geq0}
W_{N-k}[\tau_{N-k,k},P^{(N-k)}a]P^{(N-k)}f\\
&=\sum_{k\ge0}\sum_{l\ge 1}\Delta_{N+l}W_{N-k}[\tau_{N-k,k},P^{(N-k)}a]P^{(N-k)}f.
\end{align*}
It is natural to analyze the operator $\Delta_{j+s}W_{j}[\tau_{j,k},P^{(j)}a]P^{(j)}$ by considering the corresponding trilinear form as follows. 
\[\Lambda[\tau^{(2^j)}_{j,k}](P^{(j)}a,P^{(j)}f,\Delta_{j+s}g):=\int_{\mathbb R}\int_{\mathbb R ^d}\int_{\mathbb R ^d}\tau^{(2^j)}_{j,k}(\alpha,v)\Delta_{j+s}g(x)P^{(j)}f(x-v)P^{(j)}a(x-\alpha v)\,dx \,dv \, d\alpha.\]
Using $\lim_{n\to \infty}P^{(n)}b=b$ a.e. for $b \in L^\infty(\mathbb{R}^d)$ and that $\left\{\varPhi^{(n)}:n\in \mathbb{N}\right\}$ is an approximation of the identity, using the dominated convergence theorem, we deduce that,
\begin{align}\label{2.7}\lim_{M\to \infty }\Lambda[\tau^{(2^j)}_{j,k}](P^{(j+M)}P^{(j)}a,P^{(j+M)}P^{(j)}f,\Delta_{j+s}g)=\Lambda[\tau^{(2^j)}_{j,k}](P^{(j)}a,P^{(j)}f,\Delta_{j+s}g).\end{align}
On the other hand, using the mean-zero property \eqref{2.6} and the fact that $\lim_{n\to -\infty}P^{(n)}h=0$  for $h \in L^p(\mathbb{R}^d),\; 1<p<\infty$, one deduces that 
\begin{align}\label{2.8}\lim_{M\to \infty }\Lambda[\tau^{(2^j)}_{j,k}](P^{(j-M)}P^{(j)}a,P^{(j-M)}P^{(j)}f,\Delta_{j+s}g)=0.\end{align}
%{\color{red} The estimate above should be $0$.}

Using estimates \eqref{2.7}, (\ref{2.8}) and the bilinearity, we obtain
\begin{align*}
    \Lambda[\tau^{(2^j)}_{j,k}](P^{(j)}a,P^{(j)}f,\Delta_{j+s}g)&=\lim_{M \to \infty}\Big[\Lambda[\tau^{(2^j)}_{j,k}](P^{(j+M)}P^{(j)}a,P^{(j+M)}P^{(j)}f,\Delta_{j+s}g)\notag\\
	&\hspace{1.5cm}-\Lambda[\tau^{(2^j)}_{j,k}](P^{(j-M)}P^{(j)}a,P^{(j-M)}P^{(j)}f,\Delta_{j+s}g)\Big]\notag\\
    &=\lim_{M \to \infty}\sum_{m=-M+1}^{M}[\Lambda[\tau^{(2^j)}_{j,k}](P^{(j+m)}P^{(j)}a,P^{(j+m)}P^{(j)}f,\Delta_{j+s}g)-\notag\\
    & \qquad\qquad\Lambda[\tau^{(2^j)}_{j,k}](P^{(j+m-1)}P^{(j)}a,P^{(j+m-1)}P^{(j)}f,\Delta_{j+s}g)]\notag\\
    &=\lim_{M \to \infty}\sum_{m=-M+1}^{M}[\Lambda[\tau^{(2^j)}_{j,k}](\Delta_{j+m}P^{(j)}a,P^{(j+m)}P^{(j)}f,\Delta_{j+s}g)+\notag\\
    & \qquad\Lambda[\tau^{(2^j)}_{j,k}](P^{(j+m-1)}P^{(j)}a,\Delta_{j+m}P^{(j)}f,\Delta_{j+s}g)]\notag\\
    &=\sum_{m \in \mathbb{Z}}\Lambda[\tau^{(2^j)}_{j,k}](\Delta_{j+m}P^{(j)}a,P^{(j+m)}P^{(j)}f,\Delta_{j+s}g)+\notag\\
    & \qquad\sum_{m \in \mathbb{Z}}\Lambda[\tau^{(2^j)}_{j,k}](P^{(j+m-1)}P^{(j)}a,\Delta_{j+m}P^{(j)}f,\Delta_{j+s}g)\notag\\
    &=\sum_{m \in \mathbb{Z}}\Lambda[\tau^{(2^j)}_{j,k}](\Delta_{j+m}P^{(j)}a,P^{(j+m)}P^{(j)}f,\Delta_{j+s}g)+\notag\\
    &\qquad \sum_{m \in \mathbb{Z}}\Lambda[\tilde{\tau}^{(2^j)}_{j,k}](\Delta_{j+m}P^{(j)}f,P^{(j+m-1)}P^{(j)}a,\Delta_{j+s}g),\label{2.9}
\end{align*}
where $\tilde{\tau}_{j,k}(\alpha,v):=|\alpha|^{d-2}\tau_{j,k}(\alpha^{-1},\alpha v)$,  with $\|\tilde{\tau}_{j,k}\|_{\mathcal{B}_{\delta_1}(\mathbb R \times \mathbb R ^d)}\lesssim \|\tau_{j,k}\|_{\mathcal{B}_\delta(\mathbb R \times \mathbb R ^d)}$, $0<\delta_{1}<\delta/3$ and $\|\tilde{\tau}_{j,k}\|_{L^1(\mathbb R \times \mathbb R ^d)}=\|\tau_{j,k}\|_{L^1(\mathbb R \times \mathbb R ^d)}$. %{\color{red} we should say something about these estimates} 
Here in the last part we are  permuting the first and second entries, by a composition of inversion and multiplicative shear. 
\\

Let $0<\delta'<1$ and suppose that
$\sup_j\|\varsigma_j\|_{B_\delta(\mathbb{R}\times\mathbb{R}^d)}<\infty$, $\int \varsigma_j=0.$ For $b\in L^\infty(\mathbb{R}^d)$ and for $m_1\in \mathbb Z,m_2\in \mathbb N$ we define
\begin{align*}
  \big\langle \mathfrak{U}^{m_1,m_2}_{j}[b]f,g\big\rangle &:=\Lambda[\varsigma^{(2^j)}_{j}](\Delta_{j+m_1}P^{(j)}g,P^{(j+m_1-1)}P^{(j)}b,\Delta_{j+m_2}f),\\
  \big\langle \mathfrak{V}^{m_1,m_2}_{j}[b]f,g\big\rangle
 & := \Lambda[\varsigma_j^{(2^j)}](\Delta_{j+m_1}P^{(j)}b,\,P^{(j+m_1)}P^{(j)}g,\,\Delta_{j+m_2}f).
\end{align*}
Then we have the following Lemma.
\begin{lemma}\label{Lem8.3}
(i) $\sum_{j\in \mathbb Z}\mathfrak{U}^{m_1,m_2}_{j}[b]$ converges in the strong operator topology, as bounded operators on $L^2(\mathbb R^d)$. Moreover, there is $\delta ''>0$ such that, for any $n \in \{0\} \cup \mathbb N$, \begin{equation}
\left\|\sum_{j=-n}^n\mathfrak{U}^{m_1,m_2}_{j}[b]\right\|_{L^2(\mathbb{R}^d)\to L^2(\mathbb{R}^d)}\lesssim \min\{2^{-\delta''(|m_1|+m_2)}\sup_j\|\varsigma_j\|_{\mathcal{B}_{\delta'}(\mathbb{R}\times\mathbb{R}^d)},\sup_j\|\varsigma_j\|_{L^1(\mathbb{R}\times\mathbb{R}^d)}\ \}\|b\|_{L^\infty(\mathbb{R}^d)}.\label{2.10}\end{equation}
    \\(ii) $\sum_{j\in \mathbb Z}\mathfrak{V}^{m_1,m_2}_{j}[b]$ converges in the strong operator topology, as bounded operators on $L^2(\mathbb R^d)$. Moreover, there is $\delta_1''>0$ such that, for any $n \in \{0\} \cup \mathbb N$,
    
    \begin{align}
    &\left\|\sum_{j=-n}^n\mathfrak{V}^{m_1,m_2}_{j}[b]\right\|_{L^2(\mathbb{R}^d)\to L^2(\mathbb{R}^d)}\notag\\
    & \qquad \lesssim \min\{2^{-\delta_1''(|m_1|+m_2)}\sup_j\|\varsigma_j\|_{\mathcal{B}_{\delta'}(\mathbb{R}\times\mathbb{R}^d)},\sup_j\|\varsigma_j\|_{L^1(\mathbb{R}\times\mathbb{R}^d)}\log(1+\tfrac{\sup_j\|\varsigma_j\|_{\mathcal{B}_{\delta'}(\mathbb{R}\times\mathbb{R}^d)}}{\sup_j\|\varsigma_j\|_{L^1(\mathbb{R}\times\mathbb{R}^d)}}) \}\|b\|_{L^\infty(\mathbb R^d)}.\label{2.11}
    \end{align}

\end{lemma}
For a proof of the result above see Proposition 13.4. and 13.5. in \cite{MR3897981}.

Finally note that 

\begin{align*}
    \mathcal{L}_{2,N}[a]f
&=\sum_{k\ge0}\sum_{l\ge 1}\Delta_{N+l}W_{N-k}[\tau_{N-k,k},P^{(N-k)}a]P^{(N-k)}f\\
&=\sum_{m \in \mathbb Z}\sum_{k\ge0}\sum_{l\ge 1}{}^t\mathfrak{U}^{m,k+l}_{N-k,k}[a]f +\sum_{m \in \mathbb Z}\sum_{k\ge0}\sum_{l\ge 1}{}^t\mathfrak{V}^{m,k+l}_{N-k,k}[a]f,
\end{align*}
where 
\begin{align*}
  \big\langle \mathfrak{U}^{m,k+l}_{N-k,k}[a]f,g\big\rangle &:=\Lambda[\tau^{(2^{N-k})}_{N-k,k}](\Delta_{N-k+m}P^{(N-k)}g,P^{(N-k+m-1)}P^{(N-k)}a,\Delta_{N+l}f),\\
  \big\langle \mathfrak{V}^{m,k+l}_{N-k,k}[a]f,g\big\rangle
 & := \Lambda[\tau_{N-k,k}^{(2^{N-k})}](\Delta_{N-k+m}P^{(N-k)}a,\,P^{(N-k+m)}P^{(N-k)}g,\,\Delta_{N+l}f) \  \text{and}\\
  \big\langle {}^t\mathfrak{U}^{m,k+l}_{N-k,k}[a]f,g\big\rangle &:= \big\langle \mathfrak{U}^{m,k+l}_{N-k,k}[a]g,f\big\rangle.
\end{align*}
Using \eqref{2.4}, \eqref{2.10} and \eqref{2.11}, we get
uniformly over signs $\epsilon_j\in \{-1,1\}$, for any $n\in \{0\}\cup \mathbb{N}$
\begin{align*}
\left\|
\sum_{j=-n}^n\epsilon_j\mathfrak{U}^{m,s}_{j,k}[a]\right\|_{L^2(\mathbb{R}^d)\to L^2(\mathbb{R}^d)}
&\lesssim 2^{-\delta'(|m|+s)}
,\\
\left\|
\sum_{j=-n}^n\epsilon_j\mathfrak{V}^{m,s}_{j,k}[a]\right\|_{L^2(\mathbb{R}^d)\to L^2(\mathbb{R}^d)}
&\lesssim 2^{-\delta_1'(|m|+s)},
\end{align*}
and the remainder of the proof proceeds as in the preceding cases. This concludes the proof of the estimate \eqref{thm:strong22}

\section{Weak type $(1,1)$ boundedness of $T^*[a]$: Proof of \Cref{Maintheorem}} \label{sec:proofmain}
	
	We may assume that $A=1$, $\|a\|_{L^\infty(\mathbb{R}^d)}=1$. By the Calder\'on--Zygmund decomposition applied to $f$ at height $\alpha$, there exists a family $\mathcal{Q}$ of disjoint dyadic cubes such that:
	
	\begin{enumerate}
		\item $f = g + b$ with $b = \sum_{s \in \mathbb{Z}} B^s, \ \text{where } B^s:= \sum_{Q \in \mathcal{Q}:|Q|=2^{sd}} b_Q. $
		
		\item 
		$\|g\|_{L^2(\mathbb{R}^d)}^2 \lesssim \alpha\|f\|_{L^1(\mathbb{R}^d)}$.

		\item 
		$
		\int b_Q(x)\,dx = 0$ 
		and
		$\operatorname{supp}(b_Q) \subset Q. $

		\item 
		$
		\|b_Q\|_{L^1(\mathbb{R}^d)} \lesssim \alpha |Q|$,
		and
		$\sum_{Q \in \mathcal{Q}} |Q| \leq \frac{1}{\alpha} \|f\|_{L^1(\mathbb{R}^d)}. $
	\end{enumerate}
	By the $L^2$- boundedness of $T^*[a]$, the desired estimate for the good part follows immediately. 
	\[
	\left| \{ x \in \mathbb{R}^d : |T^*[a]g(x)| >  \frac{\alpha}{2}\} \right|
	\lesssim \frac{1}{\alpha^2} \|g\|_{L^2(\mathbb{R}^d)}^2
	\lesssim \frac{1}{\alpha}\|f\|_{L^1(\mathbb{R}^d)}. 
	\]
	We define the exceptional set $E := \bigcup_{Q \in \mathcal{Q}} Q^*$,
	where $Q^*$ has the same center as $Q$ and sidelength $\ell(Q^*) = 2^{10} \ell(Q)$. 
	Then,
	\begin{align*}
	|E| \leq \sum_{Q \in \mathcal{Q}} |Q^*| 
	\lesssim \frac{1}{\alpha} \|f\|_{L^1(\mathbb{R}^d)}. 
	\end{align*}
	Recall the discretization of the operator from equation \ref{discritazation} in Section~\ref{sec:proofL2}, we have 
	\begin{align*} 
	%\label{3.1}
	T^*[a] b(x)\leq Mb(x)+T^*_{dis}[a] b(x),\end{align*}
	where the operator $T^*_{dis}[a]$ is defined as
	\[T^*_{dis}[a] f(x) := \sup_{k \in \mathbb{Z}} \left| \sum_{j > k} T_jf(x) \right|.\]
    
    Therefore, it is enough to estimate the following.
	\[\left| \left\{ x \in \mathbb{R}^d \setminus E :
	|T^*_{dis}b(x)\big| > \tfrac{\alpha}{4}\right\} \right|.\]

	By support considerations, it follows that for $x\in E^c$, we have
	\begin{align*}
		\sum_{s\leq1} \sum_{j \in\mathbb{Z}}\int K_j(x-y) \int_0^1 a(t x+(1-t) y) d t\, B^{j-s}(y) d y=0.
	\end{align*}
	Let $s(\varepsilon) := 10^{10} d\, \varepsilon^{-1} \log_2(2\varepsilon^{-1})$. For each $1\leq s \leq s(\varepsilon)$, we have
	\begin{align}\label{3.2}
		\sum_{j \in \mathbb{Z}} \|T_jB^{j-s}\|_{L^1(\mathbb{R}^d)}
		&\lesssim \sum_{j \in \mathbb{Z}}\|K_j\|_{L^1(\mathbb{R}^d)} \|B^{j-s}\|_{L^1(\mathbb{R}^d)}\lesssim \|f\|_{L^1(\mathbb{R}^d)}. 
	\end{align}
	Thus, for the finitely many terms, by Chebyshev's inequality, we have that
	\begin{align*}
	%\label{3.3}
	\left| \left\{ x \in \mathbb{R}^d \setminus E :
	\sum_{1\leq s\leq s(\varepsilon)} \sum_{j \in \mathbb{Z}} |T_j B^{j-s}(x)| >  \frac{\alpha}{4}\right\} \right|
	\lesssim \frac{1}{\alpha}\|f\|_{L^1(\mathbb{R}^d)}.  \end{align*}
	Hence, it suffices to show that 
	\begin{align*}
	%\label{3.4}
	\left| \left\{ x \in \mathbb{R}^d \setminus E :
	\sum_{s>s(\varepsilon)}\sup_{k \in \mathbb{Z}} \big|\sum_{j>k} T_j B^{j-s}(x)\big| > \frac{\alpha}{4}\right\} \right|
	\lesssim \frac{1}{\alpha}\|f\|_{L^1(\mathbb{R}^d)}. 
    \end{align*}
	For a fixed $\delta > 0$ to be determined later, we choose $C>0$ such that 
	$C\sum_{s>s(\varepsilon)} \frac{1}{s^{1+\delta}} = \frac{1}{4}$. Then, a pigeon-holing argument implies that
	\begin{align*}
		&\left| \left\{ x \in \mathbb{R}^d \setminus E :
		\sum_{s>s(\varepsilon)}\sup_{k \in \mathbb{Z}}\big|\sum_{j>k} T_j B^{j-s}(x)\big| > \frac{\alpha}{4}\right\} \right|\\
		&\leq\sum_{s>s(\varepsilon)}\left| \left\{ x \in \mathbb{R}^d \setminus E :\sup_{k \in \mathbb{Z}}\big|\sum_{j>k} T_j B^{j-s}(x)\big|
		> \frac{C}{s^{1+\delta}} \,\alpha \right\} \right|. 
	\end{align*}
	Let $s > s(\varepsilon)$ be fixed. We set $	\mathcal{C}_{j,m} = 2^j([0,1]^d + m)$, where $m \in \mathbb{Z}^d$ and $j \in \mathbb{Z}. $ 
	We set 
	\[
	B^{j-s} 
	= \sum_{\substack{Q \in \mathcal{Q}:\\|Q| = 2^{(j-s)d}}} b_Q
	= \sum_{m \in \mathbb{Z}^d} \sum_{\substack{Q \in \mathcal{Q}:|Q| = 2^{(j-s)d} \\ Q \subset \mathcal{C}_{j,m}}} b_Q
	= \sum_{m \in \mathbb{Z}^d} B^{j-s,m},
	\]
	where
	\[
	B^{j-s,m} 
	:= \sum_{\substack{Q \in \mathcal{Q}:|Q| = 2^{(j-s)d}\\Q \subset \mathcal{C}_{j,m}}} b_Q. 
	\]
	Note that the quantity above is defined only for those $m \in \mathbb{Z}^d$ for which at least one such $Q$ exists. Also, we have that
	\[
	\operatorname{supp}  T_j B^{j-s,m} \subset  2^3\mathcal{C}_{j,m}. 
	\]
	We observe that for each fixed $j \in \mathbb{Z}$, the cubes 
	$\{2^3 \mathcal{C}_{j,m}\}_{m \in \mathbb{Z}^d}$ have bounded overlap, but they do not have the nesting structure. Therefore, we further subdivide the cubes
	$\{\mathcal{C}_{j,m} : j \in \mathbb{Z}, m \in \mathbb{Z}^d\}$ via the shifted dyadic grids defined as follows. For $\vec{u} \in \left\{0, \frac{1}{3}, \frac{2}{3}\right\}^d$, we define the dyadic grids 
	$\mathcal{D}^{\vec{u}}$ as
	\[
	\mathcal{D}^{\vec{u}} 
	:= \left\{ 2^j ([0,1]^d + m +(-1)^j \vec{u}) : j \in \mathbb{Z},\; m \in \mathbb{Z}^d \right\}. 
	\]
	%{\color{red} Is there no $(-1)^j$ in the definition of shifted dyadic.}
    We will require the following lemma from \cite[Lemma 2.5]{MR3065022}. 
	
	\begin{lemma}\label{lemma:shifted}
		 For any cube $\mathcal{C}$, there exists $\vec{u} \in \left\{0, \frac{1}{3}, \frac{2}{3}\right\}^d$ 
		and $R \in \mathcal{D}^{\vec{u}}$ such that
		$\mathcal{C} \subset R, 2^3 \mathcal{C} \subset \widetilde{R}$ and $3\ell(\mathcal{C}) < \ell(R) \leq 6\ell(\mathcal{C}),$ where $\widetilde{R}$ denotes the unique cube in $\mathcal{D}^{\vec{u}}$ containing $R$ with sidelength $2^3 \ell(R)$. 
	\end{lemma}
	From this lemma for each $\mathcal{C}_{j,m}$, there exists $\vec{w} \in \{0,\frac{1}{3},\frac{2}{3}\}^d$ and $\widetilde{\mathcal{C}}_{j,m} \in \mathcal{D}^{\vec{w}}$ such that $2^3 \mathcal{C}_{j,m} \subset \widetilde{\mathcal{C}}_{j,m}$ and $|\widetilde{\mathcal{C}}_{j,m}| \lesssim|\mathcal{C}_{j,m}|$, which leads to the following decomposition:
	\[
	B^{j-s,m}
	= \sum_{\vec{u} \in \{0,\frac{1}{3},\frac{2}{3}\}^d}
	B_{\vec{u}}^{j-s,m},
	\]
	where
	\[
	B_{\vec{u}}^{j-s,m}
	:=
	\begin{cases}
		B^{j-s,m}, & \text{if } \vec{u} = \vec{w}, \\
		0, & \text{otherwise}. 
	\end{cases}
	\]
	Thus, it is enough to estimate the quantity 
    \[\left| \left\{ x \in \mathbb{R}^d \setminus E :
	\sup_{k \in \mathbb{Z}} \big| \sum_{j>k} \sum_{m \in \mathbb{Z}^d}
	T_jB_{\vec{u}}^{j-s,m}(x) \big|> \frac{C}{3^ds^{1+\delta}} \,\alpha  \right\} \right|, \;\text{for every fixed}\;\vec{u} \in \{0,\tfrac{1}{3},\tfrac{2}{3}\}^d. \]
	Note that for each fixed $\vec{u} \in \{0,\frac{1}{3},\frac{2}{3}\}^d$, $
	\operatorname{supp} T_j B_{\vec{u}}^{j-s,m}
	\subset \widetilde{\mathcal{C}}_{j,m},$ 
	where $\widetilde{\mathcal{C}}_{j,m} \in \mathcal{D}^{\vec{u}}$, and any two such cubes $\widetilde{\mathcal{C}}_{j,m}$ and $\widetilde{\mathcal{C}}_{j',m'}$ are either disjoint or one is contained in another. 
	We set
	\[
	C^{\vec{u}}_{0} := \left\{ (j,m) \in \mathbb{Z} \times \mathbb{Z}^d :
	\exists\, Q \in \mathcal{Q},\ |Q| = 2^{(j-s)d},\ Q \subset \mathcal{C}_{j,m}\ \text{and} \ \widetilde{\mathcal{C}}_{j,m} \in \mathcal{D}^{\vec{u}} \right\}\] and  \[\mathfrak{X}^{\vec{u}} := \bigcup_{(j,m) \in C^{\vec{u}}_0} \widetilde{\mathcal{C}}_{j,m}. 
	\]
	Thus, $\operatorname{supp} \left( \sum_{j \in \mathbb{Z}} \sum_{m \in \mathbb{Z}^d}T_j B_{\vec{u}}^{j-s,m} \right)\subset \mathfrak{X}^{\vec{u}}$. 
	Let\[
	\mathfrak{X}^{\vec{u}}_{1}:= \left\{ x \in \mathfrak{X}^{\vec{u}} :
	\sum_{(j,m) \in C^{\vec{u}}_{0}} \chi_{\widetilde{C}_{j,m}}(x)
	> 2^{s(d+\varepsilon_0)} \right\} \,\text{and}\quad \mathfrak{X}^{\vec{u}}_{2}:=\mathfrak{X}^{\vec{u}}\setminus \mathfrak{X}^{\vec{u}}_{1},
	\]
	where $\varepsilon_0 > 0$ is a small constant. Note that on the set $\mathfrak{X}^{\vec{u}}_{2}$ we have controlled overlap in the scales with respect to the parameter s. We write
	\begin{align*}
		&\left| \left\{ x \in \mathbb{R}^d \setminus E :
		\sup_{k \in \mathbb{Z}} \big| \sum_{j>k} \sum_{m \in \mathbb{Z}^d}
		T_jB_{\vec{u}}^{j-s,m}(x) \big|
		> \frac{C}{3^ds^{1+\delta}} \,\alpha  \right\} \right|\\
		\leq&|\mathfrak{X}_1^{\vec{u}}|
		+ \left| \left\{ x \in \mathfrak{X}_2^{\vec{u}} :	\sup_{k \in \mathbb{Z}} \big| \sum_{j>k} \sum_{m \in \mathbb{Z}^d}
		T_jB_{\vec{u}}^{j-s,m}(x) \big|
		> \frac{C}{3^ds^{1+\delta}} \,\alpha  \right\}
		\right|. 	
	\end{align*}
	We note that for any $(j,m) \in C_0^{\vec{u}}$, there exists at least one $Q_{j,m} \in \mathcal{Q}$
	with $|Q_{j,m}| = 2^{(j-s)d}$ such that $Q_{j,m} \subset \mathcal C_{j,m}$. By Chebyshev's inequality, we have
	\begin{align*}
		|\mathfrak{X}_1^{\vec{u}}|
		&\leq \frac{1}{2^{s(d+\varepsilon_0)}}
		\sum_{(j,m) \in C^{\vec{u}}_0} |\widetilde{\mathcal{C}}_{j,m}|\\
		&\lesssim {2^{-s\varepsilon_0}}
		\sum_{(j,m)\in C^{\vec{u}}_0} |Q_{j,m}|\\
		&\leq \frac{2^{-s\varepsilon_0}}{\alpha} \|f\|_{L^1(\mathbb{R}^d)}. 
	\end{align*}
	Let
	$C^{\vec{u}} := \left\{ (j,m) \in C^{\vec{u}}_0 : \widetilde{\mathcal{C}}_{j,m} \not\subset \mathfrak{X}_1^{\vec{u}} \right\}.$ Then for $x \in \mathfrak{X}_2^{\vec{u}}$, $\sum_{\substack{(j,m)\in C^{\vec{u}}_0 \setminus C^{\vec{u}}},\,j>k}
		T_j B_{\vec{u}}^{j-s,m}(x)=0$. Hence, we have
	\begin{align*}
		\sum_{j>k} \sum_{m \in \mathbb{Z}^d}T_j B_{\vec{u}}^{j-s,m}(x)&=\sum_{\substack{(j,m)\in C^{\vec{u}},\, j>k}}T_j B_{\vec{u}}^{j-s,m}(x). 
	\end{align*}
	Next, we observe that for any fixed $(j,m) \in C^{\vec{u}}$, there are at most $2^{s(d+\varepsilon_0)}$ indices $(j',m') \in C^{\vec{u}}$ such that 
	$\widetilde{\mathcal{C}}_{j,m} \subset \widetilde{\mathcal{C}}_{j',m'}. $ Based on this observation, we subdivide the collection $C^{\vec{u}}$ as follows:
	\[
	C_1^{\vec{u}} := \left\{ (j,m) \in C^{\vec{u}} :
	\text{there is no other } (j',m') \in C^{\vec{u}} \text{ such that }
	\widetilde{\mathcal{C}}_{j,m} \varsubsetneqq \widetilde{\mathcal{C}}_{j',m'} \right\},
	\]
	\[
	C_2^{\vec{u}} := \left\{ (j,m) \in C^{\vec{u}} \setminus C_1^{\vec{u}} :
	\text{there is no other } (j',m') \in C^{\vec{u}} \setminus C_1^{\vec{u}}
	\text{ such that } \widetilde{\mathcal{C}}_{j,m} \varsubsetneqq \widetilde{\mathcal{C}}_{j',m'} \right\},
	\]
	\[
	\vdots
	\]
	\[
	C_l^{\vec{u}} := \left\{ (j,m) \in C^{\vec{u}} \setminus \bigcup_{i=1}^{l-1} C_i^{\vec{u}} :
	\text{there is no other } (j',m') \in C^{\vec{u}} \setminus \bigcup_{i=1}^{l-1} C_i^{\vec{u}}
	\text{ such that } \widetilde{\mathcal{C}}_{j,m} \varsubsetneqq \widetilde{\mathcal{C}}_{j',m'} \right\}. 
	\]
	This procedure terminates after at most $l = 2^{s(d+\varepsilon_0)}$ steps. 

	We claim that the following pointwise inequality is true %{\color{red} is the dependence of $k$ in RHS missing here?}
	\begin{align*}
		\left| \sum_{j>k} \sum_{m \in \mathbb{Z}^d}
		T_j B^{j-s,m}_{\vec{u}}(x) \right|
		&=
		\left| \sum_{1 \leq i \leq t_0(k,x)}
		\sum_{(j,m)\in C_i^{\vec{u}}}
		T_j B^{j-s,m}_{\vec{u}}(x) \right|  \ \text{for}\ x\in \mathfrak{X}^{\vec{u}}_2,
	\end{align*}
	%{\color{blue} what is $t_0$?}
	%To see this, we let $k \in \mathbb{Z}$ be fixed and $x \in \mathfrak{X}_2^{\vec{u}}$. If $x \notin \bigcup_{l=1}^{2^{s(d+\varepsilon_0)}} \bigcup_{(j,m)\in C_l^{\vec{u}}} \widetilde{\mathcal{C}}_{j,m},$ then $\left| \sum_{j>k} \sum_{m \in \mathbb{Z}^d}T_j B^{j-s,m}_{\vec{u}}(x) \right|= 0.$So let	$x \in \bigcup_{l=1}^{2^{s(d+\varepsilon_0)}} 
	%\bigcup_{(j,m)\in C_{l}^{\vec{u}}} \widetilde{\mathcal{C}}_{j,m}. $Let $l_0$ be the largest index such that $x \in \bigcup_{(j,m)\in C_{l_0}^{\vec{u}}} \widetilde{\mathcal{C}}_{j,m}. $Then $x \notin \bigcup_{(j,m)\in C_i^{\vec{u}}} \widetilde{K}_{j,m}$ for all $i>l_0. $
	
	%Since $\{\widetilde{\mathcal{C}}_{j,m}\}_{(j,m)\in C_{l_0}^{\vec{u}}}$ are pairwise disjoint,there exists a unique $(j_{l_0},m_{l_0})\in C_{l_0}^{\vec{u}}$ such that $x \in \widetilde{\mathcal{C}}_{j_{l_0},m_{l_0}}. $
	%So $(j_{l_0},m_{l_0}) \notin C_i^{\vec{u}}$ for all $i<l_0$. Thus for each $i \in \{1,2,\dots,l_0-1\}$  by the bounded overlap property as explained earlier there exists $(j_i,m_i)\in C_i^{\vec{u}}$ such that
	%\[
	%\widetilde{\mathcal{C}}_{j_{l_0},m_{l_0}} \subset \widetilde{\mathcal{C}}_{j_{l_0-1},m_{l_0-1}} 
	%\subset \cdots \subset \widetilde{\mathcal{C}}_{j_1,m_1}. 
	%\]

	where $t_0 \in \{1,2,\dots,l_0\}$ is the smallest index such that $j_{t_0} \leq k$, where    $l_0$ be the largest index such that $x \in \bigcup_{(j,m)\in C_{l_0}^{\vec{u}}} \widetilde{\mathcal{C}}_{j,m} $ and $(j_{t_0},m_{t_0})\in C_{t_0}^{\vec{u}}$ is the unique index such that $x \in \widetilde{\mathcal{C}}_{j_{t_0},m_{t_0}}. $
	
	%Then, using the above inclusion and the fact that $
	%2^3 3. {2^j} < \ell(\widetilde{\mathcal{C}}_{j,m}) \leq 2^36. 2^j,$
	%we have $j_t \leq k$ for all $t \in \{t_0, t_0+1, \dots, l_0\}$ and the equality follows.

	We will require a crucial $L^2$-estimate for the multiscale operator with decay in the parameter $s$. We state the result below and provide its proof in the following section.
	\begin{proposition}\label{mainlemma}
		Let $\alpha > 0$ and $\mathcal{Q}$ be a collection of cubes with disjoint interiors such that for each $Q \in \mathcal{Q}$ there exists a function $b_Q \in L^1(Q)$ satisfying $\int_Q |b_Q(x)|\,dx \leq \alpha |Q|$ and $\int_Q b_Q(x)\,dx = 0.$ Let
		$B^k := \sum_{Q \in \mathcal{Q}:|Q| = 2^{kd}} b_Q.$
		Then there exists $0<\delta_0<\frac{1}{7}$, such that for any $s>s(\varepsilon)$, we have
		\[
		\left\|\sum_{j \in \mathbb{Z}} T_j B^{j-s} \right\|_{L^2(\mathbb{R}^d)}^2
		\lesssim s^{-5-3\delta_0}\ \alpha
		\sum_{Q \in \mathcal{Q}} \|b_Q\|_{L^1(\mathbb{R}^d)}.
		\]  
	\end{proposition}
	Assuming Proposition~\ref{mainlemma}, we complete the proof of \Cref{Maintheorem}.	By Chebyshev's inequality and Rademacher-Menshov theorem, we have 
	\begin{align*}
		&\left| \left\{ x \in \mathfrak{X}^{\vec{u}}_2:
		\sup_{k \in \mathbb{Z}} \big| \sum_{j>k} \sum_{m \in \mathbb{Z}^d}
		T_j B^{j-s,m}_{\vec{u}}(x) \big|
		> \frac{C}{3^d s^{1+\delta}} \alpha \right\} \right| \\
		= &\left| \left\{ x \in \mathfrak{X}^{\vec{u}}_2:
		\sup_{1\leq l \leq 2^{s(d+\varepsilon_0)}} \big|\sum_{1 \leq i \leq l}\sum_{(j,m)\in C^{\vec{u}}_i} T_j B^{j-s,m}_{\vec{u}}(x) \big|
		> \frac{C}{3^d s^{1+\delta}} \alpha \right\} \right|\\
		\lesssim & \frac{s^{2(1+\delta)}}{\alpha^2}
		\left\|	\sup_{1 \leq l \leq 2^{s(d+\varepsilon_0)}}\big|\sum_{1 \leq i \leq l}\sum_{(j,m)\in C^{\vec{u}}_i} T_j B^{j-s,m}_{\vec{u}}\big|
		\right\|_{L^2(\mathbb{R}^d)}^2 \\
		\lesssim & \frac{s^{2(2+\delta)}}{\alpha^2}
		\sup_{\epsilon_l \in \{-1,0,1\}}
		\left\|
		\sum_{1\leq l \leq 2^{s(d+\varepsilon_0)}} 
		\sum_{(j,m)\in C_l^{\vec{u}}}
		T_j\epsilon_lB^{j-s,m}_{\vec{u}}
		\right\|_{L^2(\mathbb{R}^d)}^2 \\
		= & \frac{s^{2(2+\delta)}}{\alpha^2}	
		\sup_{\epsilon_l \in \{-1,0,1\}}
		\left\|
		\sum_{j \in \mathbb{Z}} T_j \left(
		\sum_{1\leq l \leq 2^{s(d+\varepsilon_0)}} 
		\sum_{m \in \mathbb{Z}^d : (j,m)\in C_l^{\vec{u}}}
		\varepsilon_l B^{j-s,m}_{\vec{u}}
		\right)
		\right\|_{L^2(\mathbb{R}^d)}^2\\
		\lesssim & \tfrac{s^{-1-3\delta_0+2\delta}}{\alpha} \|f\|_{L^1(\mathbb{R}^d)},
	\end{align*}
	where the final step follows from Proposition~\ref{mainlemma}. Choosing $\delta=\delta_0$ and summing in $\vec{u}$ and $s$ concludes the proof of \Cref{Maintheorem}. 
			
	\section{Proof of Proposition~\ref{mainlemma}}
	The proof exploits the microlocal decomposition of a smoothened version of the kernel as developed in      \cite{Se2014}. We prove estimates for the pieces in the decomposition and then combine them  into a single $L^2$ estimate via a careful interpolation argument motivated by the approach developed in \cite{bhojak2026endpoint}.
	
	Let
	\begin{align*}
	\ell(s) := [11\log_2 s] + 2 
	\quad \text{and} \quad
	\ell_{\varepsilon}(s) := [11\varepsilon^{-1} \log_2 s ] + 2. 
	\label{4.1}
	\end{align*}
	Let $\Phi$ be a radial nonnegative $C_c^\infty(\mathbb{R}^d)$ function supported in $\{x \in \mathbb{R}^d:|x|\leq 1\}$ and satisfying $\int \Phi(x)\,dx = 1.$
	Let $\Phi_m(x) := 2^{-md}\Phi(2^{-m}x)$. We define a smoothened version of the kernels $K_j$ as follows
	\[
	K_j^s := K_j * \Phi_{j-\ell_{\varepsilon}(s)},\quad\text{for}\;s>s(\varepsilon).
	\]
	We note that $K_j^s$ is supported in $\{x \in \mathbb{R}^d : 2^{j-2} \leq |x| \leq 2^{j+2}\}$ and it enjoys the following control in derivatives.
	\begin{align}\label{4.2}
	|\partial^\alpha K_j^s(x)|
	\leq C_\alpha 2^{-jd} 2^{-(j-\ell_\varepsilon(s))|\alpha|}. 
	\end{align}
	This is simply a consequence of the support conditions and \eqref{1.1}.	
    
    Let $\vartheta_s \in C^\infty(\mathbb{R})$ be such that $\operatorname{supp} \vartheta_s \subset (s^{-11},1-s^{-11})$  and $\vartheta_s(t) = 1$ for $t \in \left[2s^{-11},1-2s^{-11} \right]$ and the derivatives satisfy
	\begin{align*}
	%\label{4.4}
\|\vartheta_s^{(N)}\|_{L^\infty(\mathbb{R})}
	\leq C_N\, s^{11N}. 
	\end{align*} 
	Let $T_j^s$ and $P_m$ be the integral operators defined as follows
	\begin{align*}
		T_j^sf(x):&=\int K_j^s(x-y)\int \vartheta_s(t)\, a(tx + (1-t)y)\, dtf(y)dy\quad\text{and}\quad P_mf(x):=\Phi_m*f(x).
	\end{align*}
	By the mean value theorem and \eqref{1.1}, we have the following control over the error term.
	\begin{align}
		&\|K_j - K_j^s\|_{L^1(\mathbb{R}^d)}\notag\\
		\leq& \int_{|y|\leq 2^{j-\ell_\varepsilon(s)}} |\Phi_{j-\ell_\varepsilon(s)}(y)|
		\int_{2^{j-2} \leq |x| \leq 2^{j+2}}\big(\left| K(x) - K(x-y) \right| \, \left| \psi(2^{-j}(x-y)) \right| \notag\\
		&\hspace{5cm}+ |K(x)| \, \left| \psi(2^{-j}x) - \psi(2^{-j}(x-y)) \right|\big)dxdy\notag\\
		\lesssim& \int_{|y|\le 2^{j-\ell_\varepsilon(s)}} |\Phi_{j-\ell_\varepsilon(s)}(y)|\int_{2^{j-2} \le |x| \le 2^{j+2}}
		\left( |y|^\varepsilon |x|^{-d-\varepsilon} + |y|\, |x|^{-d-1} \right) dx\,dy\notag\\
		\lesssim& 2^{-\ell_\varepsilon(s)\varepsilon}\lesssim s^{-11}. \label{4.3}
	\end{align}
	We record the following lemma which is due to the above fact \eqref{4.3} and the support condition of $\vartheta_s$.
	\begin{lemma}\label{lem 3.1} The following estimate holds 
		\[
		\|T_j - T_j^s\|_{L^1(\mathbb{R}^d)\to L^1(\mathbb{R}^d)} \lesssim s^{-11}. 
		\]
	\end{lemma}
	
	We also have the following estimate. 

	\begin {lemma}\label{lem 3.2}
	Let $\alpha > 0$ and $\mathcal{Q}$ be a collection of cubes with disjoint interiors such that for each $Q \in \mathcal{Q}$ there exists a function $b_Q \in L^1(Q)$ satisfying $\int_Q |b_Q(x)|\,dx \leq \alpha |Q|$ and $\int_Q b_Q(x)\,dx = 0.$ Let
	$B^k = \sum_{Q \in \mathcal{Q}:\ell(Q) = 2^k} b_Q. $
	Then for any $s>1$,
	\[
	\left\| P_{j-s+\ell(s)} T_j^s B^{j-s} \right\|_{L^1(\mathbb{R}^d)}
	\lesssim s^{-11} \log s \, \left\| B^{j-s} \right\|_{L^1(\mathbb{R}^d)}. 
	\]
	\end{lemma}

\subsection*{A Microlocal Decomposition and Related Estimates}  Let $\frac{1}{10} < \gamma < \frac{9}{10}$. Let $\Theta_s$ be the  set of unit vectors such that if $\nu \neq \nu'$ and $\nu, \nu' \in \Theta_s$, then $|\nu- \nu'| \ge 2^{-4 - s\gamma}$ and  for every $\theta \in S^{d-1}$, there exists $\nu_s \in \Theta_s$ such that $|\theta - \nu_s| < 2^{-4 - s\gamma}$.  Moreover, we also have $\operatorname{card}(\Theta_s) \lesssim 2^{s\gamma(d-1)}.$

For $\nu\in S^{d-1}$ we take a nonnegative function $\widetilde{\chi}_{s,\nu} \in C^\infty(S^{d-1})$ such that $\widetilde{\chi}_{s,\nu}(\theta) = 1$ if $|\theta - \nu| \le 2^{-3 - s\gamma}$, 
$\widetilde{\chi}_{s,\nu}(\theta) = 0$ if $|\theta - \nu| \ge 2^{-2 - s\gamma}$,
and for each $M \in \mathbb{N}$,
$2^{-s\gamma M} \widetilde{\chi}_{s,\nu}$ is a bounded family in $C^M(S^{d-1})$. For each $\theta \in S^{d-1} $  there is at least one $\nu \in \Theta_s $ such that $|\theta - \nu| \le 2^{-4 - s\gamma}$ and thus $\widetilde{\chi}_{s,\nu}(\theta) = 1$ and the number of $\nu \in \Theta_s$ for which $\widetilde{\chi}_{s,\nu}(\theta) \neq 0$ is bounded above, uniformly in $\theta$ and $s$ by the separatedness assumption. We then define  for $\nu \in \Theta_s$,
\[
\chi_{s,\nu}(x)
:= \frac{\widetilde{\chi}_{s,\nu}\!\left(\frac{x}{|x|}\right)}
{\displaystyle \sum_{\nu' \in \Theta_s}
	\widetilde{\chi}_{s,\nu'}\!\left(\frac{x}{|x|}\right)},
\quad \text{for } x \neq 0. 
\]
Then $\sum_{\nu \in \Theta_s} \chi_{s,\nu}(x) = 1,$ for $x \neq 0,$ and we have the following
    \begin{align}\label{4.5}
|(\langle \nu ,\nabla \rangle)^M \chi_{s,\nu}(x)| \lesssim C_M |x|^{-M}, \quad \text{for } x \neq 0.
\end{align}
\begin{align}\label{4.6}
|\partial^\alpha \chi_{s,\nu}(x)| \lesssim C_\alpha 2^{s\gamma |\alpha|} |x|^{-|\alpha|}, \quad \text{for } x \neq 0.
\end{align}

Let $K_j^{s,\nu}(x) := K_j^s(x)\chi_{s,\nu}(x)$, so that 
\begin{align}\label{4.7}
\operatorname{supp} K_j^{s,\nu} \subset 
\left\{ x \in \mathbb{R}^d : 2^{j-2} \le |x| \le 2^{j+2}, \;
\left|\frac{x}{|x|} - \nu \right| \le 2^{-2 - s\gamma} \right\}.
\end{align}
We have the identity $T_j^s = \sum_{\nu \in \Theta_s} T_j^{s,\nu}$, where the operators $T_j^{s,\nu}$ are defined as
\[
T_j^{s,\nu}f(x):=\int_{\R^d} K_j^{s,\nu}(x-y)\int_0^1 \vartheta_s(t)\, a(tx+(1-t)y)\,dtf(y)dy.
\]

Let $\phi \in C^\infty(\mathbb{R})$ be such that $\phi(u) = 1$ for $|u| < \frac{1}{2}$ and $\phi(u) = 0$ for $|u| \ge 1$. We define the convolution operators $G_{s,\nu}$ as
\[ 
(\widehat{G_{s,\nu} f})(\xi)
:= \phi\left(2^{s\gamma} s^{-14}\left\langle \nu, \frac{\xi}{|\xi|} \right\rangle\right)\widehat{f}(\xi).
\]
We have the splitting of the operators $T_j^s$ given by
\begin{align*}
T_j^s
&= \sum_{\nu \in \Theta_s} G_{s,\nu} T_j^{s,\nu}
+ \sum_{\nu \in \Theta_s} (I - G_{s,\nu}) T_j^{s,\nu}.
\end{align*}
Then we have the following crucial estimates which exhibit decay in the scale-separation parameter $s$.
\begin{lemma}\label{lem 3.3}
Under the same assumptions as in Lemma \ref{lem 3.2} for $s > s(\varepsilon)$ and $\nu \in \Theta_s$, we have
\[
\left\| \sum_{j \in \mathbb{Z}} (I - P_{j-s+\ell(s)})(I - G_{s,\nu}) T_j^{s,\nu} B^{j-s} \right\|_{L^1(\mathbb{R}^d)}
\lesssim s^{-11}\, 2^{-s\gamma(d-1)} \sum_{Q \in  \mathcal{Q}}\|b_Q\|_{L^1(\mathbb{R}^d)}.
\]  
\end{lemma}
\begin{lemma}\label{lem 3.4}
 Let $\alpha > 0$ and $\mathcal{Q}$ be a collection of cubes with disjoint interiors such that for each $Q \in \mathcal{Q}$ there exists a function $b_Q \in L^1(Q)$ satisfying $\int_Q |b_Q(x)|\,dx \leq \alpha |Q|$. Let $B^k := \sum_{Q \in \mathcal{Q}:|Q|= 2^{kd}} b_Q$. For $s> s(\varepsilon)$,
\[
\left\| \sum_{\nu \in \Theta_s} \sum_{j \in \mathbb{Z}}
(I - P_{j-s+\ell(s)})\ G_{s,\nu}\ T_j^{s,\nu} B^{j-s}
\right\|_{L^2(\mathbb{R}^d)}^2
\lesssim 2^{-s\gamma}\ s^{14}\ \alpha
\sum_{Q \in \mathcal{Q}} \|b_Q\|_{L^1(\mathbb{R}^d)}.
\]
\end{lemma}
Before establishing the proofs of Lemmata \eqref{lem 3.2}-\eqref{lem 3.4}, we complete the proof of Proposition~\ref{mainlemma}. We write

\begin{align*}
\sum_{j \in \mathbb{Z}}T_j B^{j-s}
&=\sum_{\nu \in \Theta_s} \sum_{j \in \mathbb{Z}} (I - P_{j-s+\ell(s)}) G_{s,\nu}T_j^{s,\nu} B^{j-s}+\Big [\sum_{j \in \mathbb{Z}} (T_j - T_j^s) B^{j-s}\\
&+\sum_{\nu \in \Theta_s} \sum_{j \in \mathbb{Z}} (I - P_{j-s+\ell(s)})(I - G_{s,\nu}) T_j^{s,\nu}B^{j-s}+ \sum_{j \in \mathbb{Z}} P_{j-s+\ell(s)} T_j^s B^{j-s}\Big]\\
&=:A+B.
\end{align*}

We observe that
\begin{align*}
    \int |(A+B)(x)|^2 \, dx
&\leq\left(
\int_{|B(\cdot)| \le s^{5+3\delta_0}\alpha}
+ \int_{|A(\cdot)| > \frac{s^{5+3\delta_0}}{2}\alpha}
+ \int_{|(A+B)(\cdot)| > \frac{s^{5+3\delta_0}}{2}\alpha}
\right)
|(A+B)(x)|^2 \, dx\\
&=: I + II + III.
\end{align*}

\subsection*{Estimate for I} By Lemmata \eqref{lem 3.1}-\eqref{lem 3.4}  we have
\begin{align*}
I
&=\int_{|B(\cdot)| \le s^{5+3\delta_0}\alpha} |(A+B)(x)|^2 \, dx \\
&\lesssim \int |A(x)|^2 \, dx
+ \int_{|B(\cdot)| \le s^{5+3\delta_0}\alpha} |B(x)|^2 \, dx \\
&\lesssim \int |A(x)|^2 \, dx
+ s^{5+3\delta_0}\alpha \int_{|B(\cdot)| \le s^{5+3\delta_0}\alpha} |B(x)| \, dx \\
&\lesssim (2^{-\gamma s} s^{14}+s^{-6+4\delta_0})\ \alpha \sum_{Q \in \mathcal{Q}} \|b_Q\|_{L^1(\mathbb{R}^d)}.
\end{align*}
To estimate the quantities II and III, we need the following estimate:
\begin{align}
\left| \left\{ x \in \mathbb{R}^d :
\left| \sum_{j \in \mathbb{Z}} T_j B^{j-s}(x) \right| > \lambda \alpha \right\} \right|
\lesssim \frac{1}{\lambda^3 \alpha}
\sum_{Q \in \mathcal{Q}} \|b_Q\|_{L^1(\mathbb{R}^d)},\ \text{for}\ \lambda>0.\label{4.8}
\end{align}

\subsection*{Estimate for II} By Chebyshev’s inequality, Lemma~\ref{lem 3.4} and \eqref{4.8},  we have

\begin{align*}
II
&= \int_{|A(\cdot)| > \frac{s^{5+3\delta_0}}{2}\alpha}|(A+B)(x)|^2 \, dx \\
& \lesssim  \int_0^\infty \lambda \,
\min \Big(
|\{ y \in \mathbb{R}^d : |A(y)| > \frac{s^{5+3\delta_0}}{2}\alpha\}|,
\left| \left\{ y \in \mathbb{R}^d : |(A+B)(y)| > \lambda \right\} \right|
\Big) d\lambda\\
&\lesssim \|A\|_{L^2(\mathbb{R}^d)}^2
+ \alpha^2 \sum_{Q \in \mathcal{Q}} \|b_Q\|_{L^1(\mathbb{R}^d)}
\int_{\frac{s^{5+3\delta_0}}{2}\alpha}^{\infty} \frac{d\lambda}{\lambda^2} \\
&\lesssim  (2^{-\gamma s} s^{14}+s^{-5-3\delta_0})\ \alpha
\sum_{Q \in \mathcal{Q}} \|b_Q\|_{L^1(\mathbb{R}^d)}.
\end{align*}
\subsection*{Estimate for III} By \eqref{4.8}  we have

\begin{align*}
III
&= \int_{|(A+B)(\cdot)| > \frac{s^{5+3\delta_0}}{2}\alpha}|(A+B)(x)|^2 \, dx \\
& \lesssim  \int_0^\infty \lambda \,
\left| \left\{ y \in \mathbb{R}^d : |(A+B)(y)| > \max \{\lambda ,\frac{s^{5+3\delta_0}}{2}\alpha\}\right\} \right|d\lambda\\
& \lesssim s^{-5-3\delta_0}\ \alpha
\sum_{Q \in \mathcal{Q}} \|b_Q\|_{L^1(\mathbb{R}^d)}.
\end{align*}
We now prove \eqref{4.8}.
Since $K_j(x) = K(x)\psi(2^{-j}x),$ we can choose a smooth non-negative function $v_0$ supported in $\{x \in \mathbb{R}^d:2^{-2} \leq |x|\leq 2^2\}$ such that $\|v_0\|_{L^\infty(\mathbb{R}^d)}, \ \|\nabla v_0\|_{L^1(\mathbb{R}^d)} \lesssim 1$ and $|K_0(x)| \le v_0(x).$ We can construct $v_0$ such that $\frac{\psi(x)}{|x|^d} \leq v_0(x)\leq
2^d\|\psi\|_{L^\infty(\mathbb{R}^d)}.$ We take $v_j(x):= 2^{-jd} v_0(2^{-j}x).$ Then $|K_j(x)| \le \frac{\psi(2^{-j}x)}{|x|^d} \leq v_j(x)$ and similar to the estimate \eqref{3.2}, we have 

\begin{align*}%\label{4.9}
\left\| \sum_{j \in \mathbb{Z}}v_j * |B^{j-s}|\right\|_{L^1(\mathbb{R}^d)}
\lesssim \sum_{Q \in \mathcal{Q}} \|b_Q\|_{L^1(\mathbb{R}^d)}.
\end{align*}
We also have
\begin{align}
\label{4.10}
\left\| \sum_{j \in \mathbb{Z}} v_j * |B^{j-s}| \right\|_{\mathrm{BMO}}
\lesssim \alpha.
\end{align}
Using Chebyshev's inequality and interpolation we get 
\begin{align*}
\left| \left\{ x \in \mathbb{R}^d :
\left| \sum_{j \in \mathbb{Z}} T_j B^{j-s}(x) \right| > \lambda \alpha \right\} \right|
&\le \frac{1}{(\lambda \alpha)^3}
\left\|\sum_{j \in \mathbb{Z}} v_j* |B^{j-s}|\right\|_{L^3(\mathbb{R}^d)}^3 \\
&\lesssim \frac{1}{\lambda^3 \alpha}
\sum_{Q \in \mathcal{Q}} \|b_Q\|_{L^1(\mathbb{R}^d)}.
\end{align*}
Now we prove \eqref{4.10}. We fix a cube $R \subset\mathbb{R}^d$ and write
\begin{align*}
&\frac{1}{|R|} \int_R 
\left| \sum_{j \in \mathbb{Z}} v_j * |B^{j-s}|(x)
- \frac{1}{|R|} \int_R \sum_{2^j \ge \ell(R)} v_j * |B^{j-s}|(y)\,dy \right| dx \\
&\le \frac{1}{|R|} \int_R 
\left| \sum_{2^j \ge \ell(R)} v_j * |B^{j-s}|(x)
- \frac{1}{|R|} \int_R \sum_{2^j \ge \ell(R)} v_j * |B^{j-s}|(y)\,dy \right| dx \\
&\quad + \frac{1}{|R|} \int_R 
\left| \sum_{2^j < \ell(R)} v_j * |B^{j-s}|(x) \right| dx \\
&=: S_1 + S_2.
\end{align*}
By mean value theorem and local support of the functions, we get 
\begin{align*}
S_1
&\le \frac{1}{|R|^2} \int_R \int_R 
\sum_{2^j \ge \ell(R)} 
\left| v_j * |B^{j-s}|(x) - v_j * |B^{j-s}|(y) \right| dy\,dx \\
&\leq \frac{1}{|R|^2} \int_R \int_R 
\int_{\mathbb{R}^d} 
\sum_{2^j \ge \ell(R)}
\|\nabla v_j\|_{L^\infty(\mathbb{R}^d)} |x-y|\, |B^{j-s}\chi_{B(c_R, \sqrt{d} \ 2^{j+3})}(z)|\,dz\,dy\,dx\\
&\leq \frac{1}{|R|} \int_R
\sum_{2^j \ge \ell(R)} \sqrt{d} \ \ell({R})
2^{-j} \|(\nabla v_0)_j\|_{L^\infty(\mathbb{R}^d)}
\|B^{j-s} \chi_{B(c_R,\sqrt{d} \ 2^{j+3})}\|_{L^1(\mathbb{R}^d)} dx\\
& \lesssim \sum_{2^j \ge \ell(R)} 2^{-j} \, \ell(R) 2^{-jd}
\left\| B^{j-s} \, \chi_{B(c_R,\,\sqrt{d} \ 2^{j+3})} \right\|_{L^1(\mathbb{R}^d)}\\
&\lesssim \alpha,
\end{align*}
where $c_R$ is the centre of $R.$
Note that by the disjointness of the cubes $Q$
\begin{align*}
\left\| B^{j-s} \chi_{B(c_R,\, \sqrt{d}\,2^{j+3})} \right\|_{L^1(\mathbb{R}^d)}
&\lesssim \alpha \sum_{\substack{ Q \in \mathcal{Q} \ :\ \ell(Q)=2^{j-s} \\ Q \cap B(c_R,\sqrt{d}\,2^{j+3}) \neq \varnothing}} |Q| \\
&\lesssim \alpha \ 2^{jd} ,
\end{align*}
since for $y \in Q \cap B(c_R,\, \sqrt{d}\,2^{j+3})$, $z \in Q$,
\begin{align*}
|z - c_R|
&\le |z - y| + |y - c_R| \\
&\le \sqrt{d}\,2^{j-s-1} + \sqrt{d}\,2^{j+3} \\
&\lesssim 2^{j}.
\end{align*}

A similar argument yields 
\begin{align*}
    S_2
&\leq \frac{1}{|R|} \sum_{2^j <\ell(R)} \| v_j\|_{L^1(\mathbb{R}^d)}\left\| B^{j-s} \, \chi_{B(c_R,\,2^3\sqrt{d} \ \ell(R)} \right\|_{L^1(\mathbb{R}^d)}\\
&\lesssim \alpha.
\end{align*}

This concludes the proof of Proposition~\ref{mainlemma}. 

We prove Lemmata \eqref{lem 3.2}-\eqref{lem 3.4} in the subsequent sections. 
%%%%%%%%%%%%%%%%%%%%%%
\section{Proof of Lemma \ref{lem 3.2}}
%\textcolor{red}{This lemma follows exactly as in \cite{Se2014} Proposition 2.2 and noting the change in definition of $\ell(s)$ and $\ell_\varepsilon(s)$.}
Let $Q \in \mathcal{Q}$ be such that $\ell(Q)=2^{j-s}$. Then by Fubini's theorem and cancellation property of $b_Q$, we get that  
\begin{align*}
&P_{j-s+\ell(s)} T_j^s b_Q(x)\\
=& \int_{\mathbb{R}^d} \Phi_{j-s+\ell(s)}(x-w)
\int_{\mathbb{R}^d}
K_j^s(w-y)\,\int\vartheta_s(t)\,a(tw+(1-t)y)\,dt \  b_Q(y)\,dy\  dw\\
=&\int \vartheta_s(t) \int_Q b_Q(y)\,\int_{\mathbb{R}^d}
\Phi_{j-s+\ell(s)}\!\left(x-w+\frac{1-t}{t}y\right)
K_j^s(w-\frac{y}{t})\, a(tw)\,dw\,dy\,dt \\
=& \int \vartheta_s(t) \int_{\mathbb{R}^d} a(tw)
\Big[ \int_Q b_Q(y)\,
\Big( \Phi_{j-s+\ell(s)}\!\left(x-w+\frac{1-t}{t}y\right) K_j^s(w-\frac{y}{t})\\
&\qquad\qquad- \Phi_{j-s+\ell(s)}\! \left(x-w+\frac{1-t}{t}c_Q\right) \Big) K_j^s(w-\frac{c_Q}{t})\,dy \Big] dw\,dt.
\end{align*}
Let
\[ \mathcal{A}_{j,s}^{x,w,t}(y)
:= \Phi_{j-s+\ell(s)}\!\left(x-w+\frac{1-t}{t}y\right)\,K_j^s(w\frac{y}{t}).
\]
Then by the mean value theorem
\[
\left|\mathcal{A}_{j,s}^{x,w,t}(y) - \mathcal{A}_{j,s}^{x,w,t}(c_Q)\right|
\le |y - c_Q| \int_0^1 \left| \nabla \mathcal{A}_{j,s}^{x,w,t}(uy + (1-u)c_Q) \right| du.
\]
Using this estimate we have 
\begin{align*}
\big|P_{j-s+\ell(s)} T_j^s b_Q(x)\big|
&\lesssim 2^{j-s}\int |\vartheta_s(t)| \int_{\mathbb{R}^d}\int_Q \int_0^1 |a(tw) b_Q(y) \nabla \mathcal{A}_{j,s}^{x,w,t}(uy + (1-u)c_Q)| du\, dy\,dw\,dt\\
&\lesssim 2^{j-s}\int_0^1\int |\vartheta_s(t)|\int_Q |b_Q(y)| \int_{\mathbb{R}^d}\\&\quad\Big(
\frac{1-t}{t}
\left| \nabla \Phi_{j-s+\ell(s)}\!\left(x - w + \frac{1-t}{t}(uy +(1-u)c_Q)\right) \right|
|K_j^s(w-\frac{uy+(1-u)c_Q}{t})|\,\\&\quad +\frac{1}{t}
\left| \Phi_{j-s+\ell(s)}\!\left(x-w+\frac{1-t}{t}(uy+(1-u)c_Q)\right) \right|\\&
\quad \left| \nabla K_j^s\!\left(w-\frac{uy+(1-u)c_Q}{t}\right) \right|\ \Big)\ dw\ dy\ dt\ du.
\end{align*}
This gives us 
\begin{align*}
&\|P_{j-s+\ell(s)} T_j^s b_Q\|_{L^1(\mathbb{R}^d)}\\
&\lesssim 2^{j-s} \int_0^1 \int|\vartheta_s(t)|\Bigg[
\|\nabla \Phi_{j-s+\ell(s)}\|_{L^1(\mathbb{R}^d)}\frac{1-t}{t}
\int_Q |b_Q(y)| \int_{\mathbb{R}^d}\left| K_j^s\!\left(w-\frac{uy+(1-u)c_Q}{t}\right)\right|dw\, dy \\
&\qquad + \|\Phi_{j-s+\ell(s)}\|_{L^1(\mathbb{R}^d)}\frac{1}{t}
\int_Q |b_Q(y)|
\int\left|\nabla K_j^s\!\left(w-\frac{uy+(1-u)c_Q}{t}\right) \right|\ dw \ dy
\Bigg] dt\ du\\
&\lesssim 2^{j-s} \big( 2^{-(j-s+\ell(s))}\|K_j^s\|_{L^1(\mathbb{R}^d)}+\|\nabla K_j^s\|_{L^1(\mathbb{R}^d)}\big) \int\frac{|\vartheta_s(t)|}{t}\ dt\ \|b_Q\|_{L^1(\mathbb{R}^d)}.
\end{align*}
Using the $\operatorname{supp}K^s_j$ and \eqref{4.2} we have 
\[
\|\nabla K_j^s\|_{L^1(\mathbb{R}^d)}
\lesssim  2^{-(j-\ell_{\varepsilon}(s))},
\]
\[
\int |\vartheta_s(t)| \,\frac{dt}{t}
\le \int_{s^{-11}}^{1-s^{-11}} |\vartheta_s(t)| \,\frac{dt}{t}
\le \log(s^{11}-1)
\lesssim \log s.
\]
Since  $2^{\ell_{\varepsilon}(s)-s}\lesssim s^{-11},$ the estimates above yield  
\begin{align*}
    \|P_{j-s+\ell(s)} T_j^s b_Q\|_{L^1(\mathbb{R}^d)}
    &\lesssim \log s\ (2^{-\ell(s)}+2^{\ell_{\varepsilon}(s)-s})\ \|b_Q\|_{L^1(\mathbb{R}^d)}\\
    &\lesssim s^{-11}\log s\ \|b_Q\|_{L^1(\mathbb{R}^d)}.
\end{align*}
This completes the proof. 
\section{Proof of Lemma \ref{lem 3.3}}
Let $Q \in \mathcal{Q}$ , $\ell(Q)=2^{j-s}$ and $c_Q$ be the centre of $Q$. Let $\nu$ be a unit vector, and let $\pi_\nu^\perp$ denote the projection onto the orthogonal complement of $\nu$, i.e., $
\pi_\nu^\perp(x)=x-\langle x,\nu\rangle \nu.$ We may assume
$\operatorname{supp} a \subset \{ y \in \mathbb{R}^d : |\langle y - c_Q,\nu \rangle| \le 2^{j+4} d\ ,|\pi_\nu^\perp(y - c_Q)| \leq 2^{j+4-s\gamma}d\}.$
Note that with this hypothesis, $a \in L^1(\mathbb{R}^d)$ and
\begin{align}\|\widehat{a}\|_{L^\infty(\mathbb{R}^d)} \le \|a\|_{L^1(\mathbb{R}^d)} \lesssim 2^{jd- s\gamma(d-1)}.\label{6.1}\end{align}
It suffices to  prove the following estimate under this hypothesis for $s > s(\varepsilon)$,
\begin{align*}
\left\| (I - P_{j-s+\ell(s)})(I - G_{s,\nu}) T_j^{s,\nu} b_Q \right\|_{L^1(\mathbb{R}^d)}
\lesssim s^{-11}\ 2^{- s\gamma(d-1)} \|b_Q\|_{L^1(\mathbb{R}^d)}.\label{6.2}
\end{align*}
Next, we decompose the function $a$ in the frequency variable. Let $\varphi$ be a radial $C^\infty(\mathbb{R}^d)$ function such that $\varphi(\xi)=1$ for $|\xi|\le 1$ and $\varphi(\xi)=0$ for $|\xi|\ge \frac{6}{5}$. Define $\beta_k(\xi):=\varphi(2^{k}\xi)-\varphi(2^{k+1}\xi)$, then $\operatorname{supp}\beta_k \subset \{\xi \in \mathbb{R}^d : 2^{-k-1}\le |\xi|\le \frac{6}{5}2^{-k}\}$. Let $\tilde{\beta}$ be a radial $C^\infty(\mathbb{R}^d)$ function such that $\operatorname{supp}\tilde{\beta} \subset \{\xi \in \mathbb{R}^d : \frac{1}{3}\le |\xi|\le \frac{3}{2}\}$ and $\tilde{\beta}(\xi)=1$ for $\frac{1}{2}\le |\xi|\le \frac{6}{5}$. Define $\tilde{\beta}_k(\xi):=\tilde{\beta}(2^{k}\xi)$, so that $\operatorname{supp}\tilde{\beta}_k \subset \{\xi \in \mathbb{R}^d : \frac{1}{3}2^{-k} \le |\xi|\le \frac{3}{2}2^{-k}\}$. So $\tilde{\beta}_k\,\beta_k = \beta_k$ for all $k \in \mathbb{Z}.$
Define operators $V_k$, $\Lambda_k$ and $\widetilde{\Lambda}_k$ with Fourier multipliers $\varphi(2^k\cdot)$, $\beta_k$ and $\tilde{\beta}_k$ respectively. 

Note that $\Lambda_k \widetilde{\Lambda}_k = \Lambda_k$ for all $k \in \mathbb{Z}$ and for every $m \in \mathbb{Z}$,
\[
\varphi(2^m \xi) + \sum_{k<m} \beta_k(\xi)
= \lim_{n \to \infty} \varphi(2^{m-n}\xi)
= \varphi(0)=1.
\]
Consequently, $$I = V_m + \sum_{k<m} \Lambda_k.$$
For a fixed $y \in Q$, we define the operator $\mathcal{K}^{s,\nu}_{j,y}$ by
\[
\mathcal{K}^{s,\nu}_{j,y}[a](x)
:= K_j^{s,\nu}(x-y)
\int_0^1 \vartheta_s(t)\, a(tx+(1-t)y)\,dt.
\]
This gives us 
\begin{align*}\label{6.3}
T_j^s b_Q(x)
= \int_Q b_Q(y) \mathcal{K}^{s,\nu}_{j,y}[a](x)\ dy.
\end{align*}

Note that
\[
\sum_{k} \beta_k(\xi)
= \sum_{k} \big(\varphi(2^k\xi)-\varphi(2^{k+1}\xi)\big)
= \lim_{N\to\infty} \big(\varphi(2^{-N}\xi)-\varphi(2^{N+1}\xi)\big)
=
\begin{cases}
1, & \xi \neq 0, \\
0, & \xi = 0.
\end{cases}
\]
Using the above decompositions, we write 
\begin{align}
   & (I - G_{s,\nu})(I - P_{j-s+\ell(s)}) T_j^{s,\nu} b_Q\notag\\
&= \sum_{k_1} \Lambda_{k_1} (I - G_{s,\nu}) \widetilde{\Lambda}_{k_1} (I - P_{j-s+\ell(s)}) \int_Q b_Q(y) \mathcal{K}^{s,\nu}_{j,y}[a]\ dy.\label{6.4}
\end{align}
\begin{align}a = V_{j-s+\ell(s)} a
+ \sum_{k_2 < j-s+\ell(s)} \Lambda_{k_2} a. \label{6.5}
\end{align}
Next, in order to obtain the desired estimates for the terms in the above, we establish three auxiliary Lemmata.
\begin{lemma}\label{lem6.1} The following estimate holds. 
    \[
\left\| \int_Q b_Q(y)\mathcal{K}^{s,\nu}_{j,y}\!\left[ V_{j-s+\ell(s)} a \right]\,dy \right\|_{L^1(\mathbb{R}^d)}
\lesssim s^{-11}2^{- s\gamma(d-1)}  \|b_Q\|_{L^1(\mathbb{R}^d)}.
\]
\end{lemma}
\begin{proof}
By the cancellation property of $b_Q$, we have 
\begin{align*}
&\left\| \int b_Q(y)\mathcal{K}^{s,\nu}_{j,y}\!\left[ V_{j-s+\ell(s)} a \right]\,dy \right\|_{L^1(\mathbb{R}^d)}\\ 
\le& \int_{\mathbb{R}^d} \int_Q |b_Q(y)|
\left| \mathcal{K}^{s,\nu}_{j,y}\!\left[ V_{j-s+\ell(s)} a \right](x)
- \mathcal{K}^{s,\nu}_{j,c_Q}\!\left[ V_{j-s+\ell(s)} a \right](x) \right| dy\,dx \\
=&\int_Q |b_Q(y)| \int_{\mathbb{R}^d} 
\left| \mathcal{K}^{s,\nu}_{j,y}\!\left[ V_{j-s+\ell(s)} a \right](x)
- \mathcal{K}^{s,\nu}_{j,c_Q}\!\left[ V_{j-s+\ell(s)} a \right](x) \right| dx\,dy.
\end{align*}
For $ y \in Q $ we estimate
\begin{align*}
&\int_{\mathbb{R}^d} \Big|
\mathcal{K}^{s,\nu}_{j,y}\!\left[ V_{j-s+\ell(s)} a \right](x)
- \mathcal{K}^{s,\nu}_{j,c_Q}\!\left[ V_{j-s+\ell(s)} a \right](x)
\Big| dx \\
&\leq\int_{\mathbb{R}^d} \Big|
\big( K_j^{s,\nu}(x-y) - K_j^{s,\nu}(x-c_Q) \big)
\int_0^1\vartheta_s(t)\, V_{j-s+\ell(s)}a(tx+(1-t)y)\,dt \\
&\qquad + K_j^{s,\nu}(x-c_Q)
\int_0^1\vartheta_s(t)
\big( V_{j-s+\ell(s)}a(tx+(1-t)y)
- V_{j-s+\ell(s)}a(tx+(1-t)c_Q) \big)\,dt
\Big| dx \\
&\leq \mathcal{E}_1(y) + \mathcal{E}_2(y),
\end{align*}
where 
\begin{align*}
&\mathcal{E}_1(y)
:= \|V_{j-s+\ell(s)}a\|_{L^{\infty}(\mathbb{R}^d)}\int_{\mathbb{R}^d} \left|
K_j^{s,\nu}(x-y) - K_j^{s,\nu}(x-c_Q)
\right| dx\quad\text{and}\\
&\mathcal{E}_2(y)
:= \int_{\mathbb{R}^d} \left|
K_j^{s,\nu}(x-c_Q)\,\int_0^1\vartheta_s(t)
\big(
V_{j-s+\ell(s)}a(tx+(1-t)y)
- V_{j-s+\ell(s)}a(tx+(1-t)c_Q)
\big)\ dt
\right| dx.
\end{align*}
\subsection*{Estimate of $\mathcal{E}_1$} By the mean value theorem we have 
\[
\int_{\mathbb{R}^d} \left| K_j^{s,\nu}(x-y) - K_j^{s,\nu}(x-c_Q) \right| dx
\le|y - c_Q| \, \|\nabla K_j^{s,\nu}\|_{L^1(\mathbb{R}^d)}.
\]
By \eqref{4.2},\eqref{4.6}
\[
|\nabla K_j^{s,\nu}(x)|
\lesssim 2^{-jd}\,2^{\ell_{\varepsilon}(s)-j} + 2^{-jd}\,2^{s\gamma}|x|^{-1}.
\]
Since 
\begin{align*}
\int_{\operatorname{supp} K_j^{s,\nu}} |x|^{-1}\,dx
&\le \int_{\{\theta \in \mathbb{S}^{d-1} : |\theta - \nu| \le 2^{-2-s\gamma}\}} \int_{2^{j-2}}^{2^{j+2}} \frac{r^{d-1}}{r}\,dr\ d\theta\\
&\lesssim 2^{j(d-1)} 2^{-s\gamma(d-1)}.
\end{align*}
and 
\[
\int_{\operatorname{supp} K_j^{s,\nu}} dx \;\lesssim\; 2^{jd}\,2^{-s\gamma(d-1)}.
\]
We get that 
\[
\|\nabla K_j^{s,\nu}\|_{L^1(\mathbb{R}^d)}
\lesssim \big( 2^{\ell_{\varepsilon}(s)-j} + 2^{s\gamma - j} \big)\,2^{-s\gamma\ (d-1)}.
\]
Next, for $k\in \mathbb{Z}$ denote $\check{\varphi}_k:=2^{-kd} \check{\varphi}(2^{-k}\cdot)$ and observe that 
\begin{align*}
\|V_k a\|_{L^\infty(\mathbb{R}^d)}
&\le \|a\|_{L^\infty(\mathbb{R}^d)} \,\|\check{\varphi}_k\|_{L^1(\mathbb{R}^d)} \lesssim 1.
\end{align*}
Therefore, for $s > s(\varepsilon)$ and $\gamma > 1/10$, we have $2^{\ell_\varepsilon(s)} \lesssim 2^{\gamma s}$ and 
\begin{align*}
|\mathcal{E}_1(y)|
&\le |y - c_Q|\big(2^{\ell_{\varepsilon}(s)-j} + 2^{s\gamma - j}\big)\,2^{-s\gamma(d-1)} \\
&\lesssim \big(2^{\ell_{\varepsilon}(s)-s} + 2^{s\gamma - s}\big)\,2^{-s\gamma(d-1)}\\
&\lesssim s^{-11}\ 2^{-s\gamma(d-1)}.
\end{align*}
\subsection*{Estimate for $\mathcal{E}_2$} By the mean value theorem, we get that 
\begin{align*}
&\left| V_{j-s+\ell(s)}a(tx+(1-t)y)
- V_{j-s+\ell(s)}a(tx+(1-t)c_Q) \right|\\
\leq&\int_{\mathbb{R}^d} \big|\check{\varphi}_{j-s+\ell(s)}(tx+(1-t)y-z)
 - \check{\varphi}_{j-s+\ell(s)}(tx+(1-t)c_Q-z)\big|dz\\
\leq &|y-c_Q| \, \|\nabla \check{\varphi}_{j-s+\ell(s)}\|_{L^1(\mathbb{R}^d)}.
\end{align*}
Since $\|\vartheta_s\|_{L^1(\mathbb{R}^d)} \leq 1$, we get that 
\begin{align*}
\mathcal{E}_2 &\le \|K_j^{s,\nu}\|_{L^1(\mathbb{R}^d)} \, |y-c_Q| \,
\|\nabla \check{\varphi}_{j-s+\ell(s)}\|_{L^1(\mathbb{R}^d)} \\
&\lesssim 2^{-s\gamma(d-1)}  \ 2^{j-s} \ 2^{s-j-\ell(s)} \\
&\lesssim s^{-11}\ 2^{- s\gamma(d-1)}.
\end{align*}
Finally, the lemma follows by integrating in $y$. This completes the proof.
\end{proof}
\begin{lemma}\label{lem6.2}Let $y \in Q$ and the function $a$ as in \eqref{6.1}.
\begin{enumerate}[(i)]
    \item Let $k_1> k_2+\ell(s)+10.$ Then
\begin{align*}\big\| \Lambda_{k_1}\mathcal{K}_{j,y}^{s,\nu}[ \Lambda_{k_2} a ] \big\|_{L^1(\mathbb{R}^d)}
&\leq C_N\,2^{-s\gamma(d-1)}\,
\min\!\left\{1,\; 2^{11d+11N}\,2^{ s\gamma}\,2^{(k_2-j+s\gamma)N}\right\}.
\end{align*}
\item Let $k_1<k_2-10$. Then
\begin{align*}\big\| \Lambda_{k_1}\mathcal{K}_{j,y}^{s,\nu}[ \Lambda_{k_2} a ] \big\|_{L^1(\mathbb{R}^d)}
+ \big\| \Lambda_{k_1}\mathcal{K}_{j,y}^{s,\nu}[ V_{k_2} a ] \big\|_{L^1(\mathbb{R}^d)} 
&\le C_N\,2^{-s\gamma(d-1)}\,
\min\!\left\{1,\; 2^{s\gamma}\,2^{(k_1-k_2)d}\,2^{(k_1-j+s\gamma)N}\right\}.
\end{align*}
\end{enumerate}
\end{lemma}
\begin{proof}
Note that $\mathcal{K}_{j,y}^{s,\nu}[a](x)
= K_j^{s,\nu}(x-y)\int \vartheta_s(t)\,a(tx+(1-t)y)\,dt.$ Therefore, we get that 
\begin{align*}
 \|\mathcal{K}_{j,y}^{s,\nu}[a]\|_{L^1(\mathbb{R}^d)}
&\le \|a\|_{L^\infty(\mathbb{R}^d)} \|K_j^{s,\nu}\|_{L^1(\mathbb{R}^d)}.
\end{align*}
For any $k \in \mathbb{Z},$ we have 
$\|\Lambda_k g\|_{L^\infty(\mathbb{R}^d)}
\lesssim \|g\|_{L^\infty(\mathbb{R}^d)}$ and 
$\|\Lambda_k g\|_{L^1(\mathbb{R}^d)}
\lesssim \|g\|_{L^1(\mathbb{R}^d)}$. Consequently, 
\begin{align*}
\|\Lambda_{k_1}\mathcal{K}_{j,y}^{s,\nu}[\Lambda_{k_2} a]\|_{L^1(\mathbb{R}^d)}
&\lesssim \|\mathcal{K}_{j,y}^{s,\nu}[\Lambda_{k_2} a]\|_{L^1(\mathbb{R}^d)} \\
&\lesssim\|\mathcal{K}_{j}^{s,\nu}\|_{L^1(\mathbb{R}^d)}
\|\Lambda_{k_2} a\|_{L^\infty(\mathbb{R}^d)} \\
&\lesssim 2^{-s\gamma(d-1)}.
\end{align*}
Similar estimate holds for the second term in $(ii)$.

Further by applying the Fourier inversion formula, Tonelli's theorem and integration by parts, we get that 
\begin{align*}
&\Lambda_{k_1}\mathcal{K}_{j,y}^{s,\nu}[\Lambda_{k_2} a](x)\\
&= \frac{1}{(2\pi)^{2N_2}}\int \vartheta_s(t)
\int_{\mathbb{R}^d}\beta_{k_2}(\eta)\,\hat{a}(\eta)\,\int_{\mathbb{R}^d}\Big[\int_{\mathbb{R}^d}
e^{2\pi i  (\langle x-w,\xi\rangle+ \langle tw+ (1-t)y,\eta\rangle)} \\
&\quad \times 
\frac{\left(I - \frac{2^{-2k_1}}{4\pi^2}\Delta_\xi\right)^{N_1}
\Big( \beta_{k_1}(\xi)\,|\xi - t\eta|^{-2N_2} \Big)\,(-\Delta_w)^{N_2} K_j^{s,\nu}(w-y)}{\big(1+2^{-2k_1}|x-w|^2\big)^{N_1}}dw \Big]d\xi
\,d\eta\,dt.
\end{align*}
Note that for $t \in \operatorname{supp} \vartheta_s, \eta \in \operatorname{supp}\beta_{k_2}$ and $\xi \in \operatorname{supp}\beta_{k_1}$, we have that $2^{-k_2-1} \le |\eta| \le \frac{6}{5}2^{-k_2}$ and  $2^{-k_1-1} \le |\xi| \le \frac{6}{5}2^{-k_1}$. Therefore, 
\begin{align*}
|\xi - t\eta|
&\gtrsim \mathcal{C}(k_1,k_2,s)
:=
\begin{cases}
2^{-k_2-\ell(s)}, & \text{if } k_1 > k_2 + \ell(s) + 10, \\[4pt]
2^{-k_1-2}, & \text{if } k_1 < k_2 - 10.
\end{cases}
\end{align*}
Also by \eqref{4.3} ,\eqref{4.6} and \eqref{4.7}
\[
\|(-\Delta_w)^{N_2} K_j^{s,\nu}\|_{L^1(\mathbb{R}^d)}
\lesssim 2^{-s\gamma(d-1)}\,2^{2N_2(\ell_\varepsilon(s)-j)}.\]
Now since $\left(2^{-k_1}\partial_\xi\right)^{N_3} \beta_{k_1}= O(1)$, we get
\[
\left|
\left(I - \frac{2^{-2k_1}}{4\pi^2}\Delta_\xi\right)^{N_1}
\big[ \beta_{k_1}(\xi)\,|\xi - t\eta|^{-2N_2} \big]
\right|
\lesssim [\mathcal{C}\big(k_1,k_2,s\big)]^{-2N_2}.
\]
Thus 
\begin{align*}
\|\Lambda_{k_1}\mathcal{K}_{j,y}^{s,\nu}[\Lambda_{k_2} a]\|_{L^1(\mathbb{R}^d)}
&= \int_{\mathbb{R}^d}
\big|\Lambda_{k_1}\mathcal{K}_{j,y}^{s,\nu}[\Lambda_{k_2} a](x)\big|\,dx \\
&\le \frac{1}{(2\pi)^{2N_2}}
\int_{\mathbb{R}^d}\int |\vartheta_s(t)|
\int_{\mathbb{R}^d}
|\beta_{k_2}(\eta)|\,|\hat{a}(\eta)|\int_{\mathbb{R}^d} \\
&\quad \times \int_{\mathbb{R}^d}
\left|
\left(I - \frac{2^{-2k_1}}{4\pi^2}\Delta_\xi\right)^{N_1}
\big[\beta_{k_1}(\xi)\,|\xi - t\eta|^{-2N_2}\big]
\right| \\
&\quad 
\frac{\left|(-\Delta_w)^{N_2} K_j^{s,\nu}(w-y)\right|}{\big(1+2^{-2k_1}|x-w|^2\big)^{N_1}} dw\, d\xi\,d\eta\,dt\,dx.
\end{align*}
 Take $N_1=[\frac{d}{2}]+1$. We integrate in $\eta$ and use that the measure of the support of $\beta_{k_2}$ is $O(2^{-k_2 d})$.
Then we integrate in $x$ and $\xi$ and use that
\[
\int_{\operatorname{supp}(\beta_{k_1})}
\int_{\mathbb{R}^d}
\left(1 + 2^{-2k_1}|x - w|^2\right)^{-N_1}
\,dx\,d\xi
= O(1).
\]
Using \eqref{6.1} we then get
\begin{align*}
\|\Lambda_{k_1}\mathcal{K}_{j,y}^{s,\nu}[\Lambda_{k_2} a]\|_{L^1(\mathbb{R}^d)}
&\lesssim_{N_2}\,2^{-k_2 d}\,\|\hat{a}\|_{L^{\infty}(\mathbb{R}^d)}\,
\|(-\Delta)^{N_2} K_j^{s,\nu}\|_{L^1(\mathbb{R}^d)}\,[\mathcal{C}(k_1,k_2,s)]^{-2N_2}.
\end{align*}
\[
\lesssim_{N_2}
\begin{cases}
2^{d\ell(s)-s\gamma(d-2)}\,
2^{(2N_2-d)(k_2-j+\ell(s)+s\gamma)} 
& \text{if } k_1 > k_2 + \ell(s) + 10, \\[6pt]
2^{-s\gamma(d-2)}\,
2^{(2N_2-d)(k_1-j+s\gamma)}\,2^{(k_1-k_2)d} 
& \text{if } k_1 < k_2 - 10.
\end{cases}\]
\[
\lesssim_{N}\,2^{-s\gamma(d-1)}
\begin{cases}
s^{11(N+d)}\,2^{s\gamma}\,2^{(k_2-j+s\gamma)N}
& \text{if } k_1 > k_2+\ell(s)+10, \\[6pt]
2^{s\gamma}\,2^{(k_1-k_2)d}\,2^{(k_1-j+s\gamma)N}
& \text{if } k_1 < k_2-10.
\end{cases}
\]
Here we have taken $N=2N_2-d$ and $
2^{\ell(s)(N-d)}
\le 2^{([11\log_2 s] + 2)(N+d)}
\lesssim s^{11(N+d)}.$ The second term in $(ii)$ is estimated in the exact same way. This completes the proof.
\end{proof}

\begin{lemma}
\label{lem 6.3} Let $k_2 - 10 \le k_1 \le k_2 + \ell(s) + 10.$ Then for every $y\in Q$, the following estimate holds
\begin{align*}
 \big\|\Lambda_{k_1}(I - G_{s,\nu})\,\mathcal{K}_{j,y}^{s,\nu}[ \Lambda_{k_2} a ] \big\|_{L^1(\mathbb{R}^d)}
&\le C_N\,2^{-s\gamma(d-1)}\,
\min\,\left\{1,\; s^{\frac{11(N+d)}{\varepsilon}}
\,2^{(d+3)s\gamma}\,2^{(k_1 - j + s\gamma)N}\right\}.
\end{align*}
\end{lemma}
\begin{proof}
We have \begin{align}\|\Lambda_{k_1}(I - G_{s,\nu}) \|_{{L^1(\mathbb{R}^d)}\to {L^1(\mathbb{R}^d)}}\lesssim 1 \label{6.6}\end{align} uniformly in $s,\nu$ and $k_1$.
Indeed, let $\lambda_s=2^{s\gamma}s^{-14}$; if $\lambda_s<\tfrac12$ then $G_{s,\nu}=I$, so assume $\lambda_s\ge\tfrac12$.
Since $\Lambda_{k_1}(I-G_{s,\nu})=\Lambda_{k_1}-\Lambda_{k_1}G_{s,\nu}$ and $\|\check{\beta_{k_1}}\|_{L^1(\mathbb{R}^d)}=\|\check\beta_0\|_{L^1(\mathbb{R}^d)}$, it suffices to bound the kernel of $\Lambda_{k_1}G_{s,\nu}$ in $L^1$, uniformly in $s,\nu,k_1$.
Compose its symbol $\beta_{k_1}(\xi)\,\phi(\lambda_s\langle\nu,\xi/|\xi|\rangle)$ with the linear map $\mathcal A\xi=2^{-k_1}\big(\pi_\nu^\perp\xi+\lambda_s^{-1}\langle\xi,\nu\rangle\nu\big)$; this gives
$\beta_0\big(\pi_\nu^\perp\xi+\lambda_s^{-1}\langle\xi,\nu\rangle\nu\big)\,\phi\big(\langle\xi,\nu\rangle/|\pi_\nu^\perp\xi+\lambda_s^{-1}\langle\xi,\nu\rangle\nu|\big)$,
a function supported in $B(0,2)$ all of whose derivatives are bounded uniformly in $s,\nu,k_1$.
Since $\mathcal A$ does not change the $L^1$ norm of a kernel, and $\|\check{g}\|_{L^1(\mathbb{R}^d)}\lesssim\sum_{|\alpha|\le d}\|\partial^\alpha g\|_{L^2(\mathbb{R}^d)}$, \eqref{6.6} follows.
\\ So $\|\Lambda_{k_1}(I - G_{s,\nu}) \mathcal{K}_{j,y}^{s,\nu}[\Lambda_{k_2} a]\|_{L^1(\mathbb{R}^d)} \lesssim 2^{-s\gamma(d-1)}$ similarly as before. %We again assume (\ref{6.1}). 
Define the convolution operator $S_{n,\nu}$ by
\[
\widehat{S_{s,\nu} g}(\eta)
= \phi\!\left(2^{s\gamma}s^{-2}\,\langle \nu,\frac{\eta}{|\eta|}\rangle\right)\hat{g}(\eta),
\]
and split $a = S_{s,\nu}a + (I - S_{s,\nu})a$.\\

Note that the claimed estimate in Lemma~\ref{lem 6.3} has two parts. 

Let us first  consider the first part of the estimate. Let 
$
b_{k_1,s,\nu}(\xi)
:= \beta_{k_1}(\xi)\left(1 - \phi\!\left(2^{s\gamma}\,s^{-14}\langle \nu,\frac{\xi}{|\xi|}\rangle \right)\right).
$
Applying Fourier inversion and Fubini's theorem, we have 
\begin{align*}
&\Lambda_{k_1}(I-G_{s,\nu})\mathcal{K}_{j,y}^{s,\nu}[\Lambda_{k_2}S_{s,\nu}a](x)\\
=& \int \vartheta_s(t)
\int_{\mathbb{R}^d}\int_{\mathbb{R}^d}
b_{k_1,s,\nu}(\xi)\,
\beta_{k_2}(\eta)\,
\phi\!\left(2^{s\gamma}s^{-2}\langle \nu,\frac{\eta}{|\eta|}\rangle\right)
\hat{a}(\eta) \\
&\qquad \times
\int_{\mathbb{R}^d}e^{2\pi i \langle t w + (1-t)y,\,\eta\rangle}
e^{2\pi i \langle x-w,\,\xi\rangle}
K_j^{s,\nu}(w-y)
\,dw\,d\xi\,d\eta\,dt .
\end{align*}
Now, observe that if $\xi \in \operatorname{supp} b_{k_1,s,\nu}$, then $2^{-k_1-1} \le |\xi| \le \frac{6}{5}\,2^{-k_1}$ and $\big|\langle \frac{\xi}{|\xi|}, \nu \rangle\big|
\ge 2^{-s\gamma-1}s^{14}$. Also, for $\eta \in \operatorname{supp}\beta_{k_2}\cdot\phi\!\left(2^{s\gamma}s^{-2}\langle \nu, \frac{\cdot}{|\cdot|} \rangle\right)\hat{a}$ we have that 
 $2^{-k_2-1} \le |\eta| \le \frac{6}{5}\,2^{-k_2}$,\ $
\big|\langle \frac{\eta}{|\eta|}, \nu \rangle\big|
\le 2^{-s\gamma}s^2.$ This yields 
$$\tfrac{|\eta|}{|\xi|}
\leq \tfrac{6}{5}\ 2^{-k_2}\cdot 2^{k_1+1}\leq \tfrac{2\cdot6}{5}2^{(k_1-k_2)}.$$
Thus, for $t\in \operatorname{supp}\vartheta_s$ and since $s>10^{10}$, we get that  
\begin{align}
\big|\langle \xi - t\eta, \nu \rangle\big|
&\ge |\xi|\left(
\big|\langle \frac{\xi}{|\xi|}, \nu \rangle\big|
- \frac{|\eta|}{|\xi|}\big|\langle \frac{\eta}{|\eta|}, \nu \rangle\big|
\right) \notag\\
&\ge |\xi|\left(
2^{-s\gamma-1}\,s^{14}
- \frac{2\cdot6}{5}2^{(k_1-k_2)}\,2^{-s\gamma}s^2
\right) \notag\\
&\ge |\xi|\,2^{-s\gamma-1}\left(
s^{14} - 8\cdot 2^{\ell(s)+10}\,s^{2}
\right)\notag\\
&\gtrsim 2^{-k_1-s\gamma}\ s^{14}.\label{6.7}
\end{align}
Now, by integration by parts, as in the proof of Lemma \ref{lem6.2}, except that here we use the directional derivative $\langle \nu, \nabla_w \rangle$ instead of $\Delta_w$, 
we have 
\begin{align*}
&\Lambda_{k_1}(I-G_{s,\nu})\mathcal{K}_{j,y}^{s,\nu}[\Lambda_{k_2}S_{s,\nu}a](x)\\
&= \frac{1}{(2\pi i)^{N_2}}
\int \vartheta_s(t)
\int_{\mathbb{R}^d}
\beta_{k_2}(\eta)\,
\phi\!\left(2^{s\gamma}s^{-2}\langle \nu,\frac{\eta}{|\eta|}\rangle\right)
\hat{a}(\eta)\int_{\mathbb{R}^d} \int_{\mathbb{R}^d}\\
&
e^{2\pi i (\langle x-w,\xi\rangle+\langle tw+(1-t)y,\eta\rangle)}
\frac{\left(I - \frac{2^{-2k_1}}{4\pi^2}\Delta_\xi\right)^{N_1}
\big[\frac{b_{k_1,s,\nu}(\xi)}{\langle \xi - t\eta, \nu \rangle^{N_2}}\big]}{\big(1-2^{-2k_1}|x-w|^2\big)^{N_1}} \\
&\quad \times\langle \nu,\nabla_w\rangle^{N_2} K_j^{s,\nu}(w-y)\,dw\ d\xi\,d\eta\,dt.
\end{align*}
Since  $\left|\partial_\xi^{N_3} b_{k_1,s,\nu}(\xi)\right| \le C_{N_3}\,(2^{s\gamma}s^{-14})^{N_3}\,2^{-k_1 N_3}$, we get 
\begin{align}\label{6.8}\left|\left(I-\frac{2^{-2k_1}\Delta_\xi}{4\pi^2}\right)^{N_1}\!\left[\frac{b_{k_1,s,\nu}(\xi)}{\langle \xi - t\eta, \nu \rangle^{N_2}}\right]\right| \le C_{N_1}\,(2^{s\gamma}s^{-14})^{2N_1}\,(2^{-(k_1+s\gamma)}s^{14})^{-N_2}.\end{align}
Also using \eqref{4.3},\eqref{4.5} and \eqref{4.7}, gives us 
\begin{align*}
\|\langle \nu,\nabla_w\rangle^{N_2} K_j^{s,\nu}\|_{L^1(\mathbb{R}^d)}
&\le C_{N_2}\,2^{(\ell_\varepsilon(s)-j)N_2}\,2^{-s\gamma(d-1)}.
\end{align*}
We assume $2N_1>d$, integrate in $x$ and $\xi$, use \eqref{6.1}, \eqref{6.7} and use that the measure of the support of $\beta_{k_2}$ is $O(2^{-k_2 d})$. Then we get 
\begin{align*}
&\big\|\Lambda_{k_1}(I-G_{s,\nu})\mathcal{K}_{j,y}^{s,\nu}[\Lambda_{k_2}S_{s,\nu}a]\big\|_{L^1(\mathbb{R}^d)}\\
&\lesssim_{N_1,N_2}
\left(2^{s\gamma}s^{-14}\right)^{2N_1}
\left(2^{-(k_1+s\gamma)}s^{14}\right)^{-N_2}\|\hat{a}\|_{L^\infty(\mathbb{R}^d)}
\,2^{-k_2 d}\,2^{(\ell_\varepsilon(s)-j)N_2}\,2^{-s\gamma(d-1)}\\
&\lesssim_{N_1,N_2}
2^{-s\gamma(d-1)}\,
s^{-14(2N_1+N_2)}\,
2^{\ell_\varepsilon(s)N_2}\,
2^{(k_1-k_2)d}\,
2^{(k_1-j+s\gamma)(N_2-d)}\,
2^{s\gamma(2N_1+1)} \\
&\lesssim_{N}
2^{-s\gamma(d-1)}\,
s^{11(N+d)}\,
2^{s\gamma(d+3)}\,
2^{(k_1-j+s\gamma)N}.
\end{align*}
Here, we have taken
$N_1 \in \left\{\frac{d+1}{2},\,\frac{d+2}{2}\right\}$ and $
N:=N_2-d$.
This completes the first estimate. 

Next, we consider the second part. Let 
\[\tilde{b}_{k,s,\nu}(\eta)
:= \beta_{k}(\eta)\left(1-\phi\!\left(2^{s\gamma}\ s^{-2}\langle\nu,\frac{\eta}{|\eta|}\rangle\right)\right).\]
Then by Fourier inversion and Tonelli's theorem, we have 
\begin{align*}
&\Lambda_{k_1}(I-G_{s,\nu})\mathcal{K}_{j,y}^{s,\nu}[\Lambda_{k_2}(I-S_{s,\nu})a](x)\\
=&\int_{\mathbb{R}^d} K_j^{s,\nu}(w-y)
\int_{\mathbb{R}^d}\int_{\mathbb{R}^d}
b_{k_1,s,\nu}(\xi)\,\tilde{b}_{k_2,s,\nu}(\eta)\,\hat{a}(\eta) \\
&\qquad \times
\left[\int \vartheta_s(t)\,
e^{2\pi i (\langle x-w,\xi\rangle+\langle tw+(1-t)y,\eta\rangle})\,dt\right]
\,d\xi\,d\eta\,dw.
\end{align*}
Now, observe that for $w-y \in \operatorname{supp}K_j^{s,\nu}$ we have that  $\left|\frac{w-y}{|w-y|} - \nu\right| \le 2^{-2-s\gamma}$. Also, for $\eta \in \operatorname{supp}\tilde{b}_{k_2,s,\nu}$, $|\langle \nu,\eta\rangle| \ge 2^{-s\gamma-1}\,s^2|\eta|$ holds. Thus, 
\begin{align}\
|\langle w-y,\eta\rangle|
&=\nonumber \big |\big\langle \frac{w-y}{|w-y|}-\nu,\eta\big\rangle
+ \langle \nu,\eta\rangle\big| |w-y| \\
&\ge \nonumber|w-y|\left( |\langle \nu,\eta\rangle|
- \left|\frac{w-y}{|w-y|}-\nu\right|\,|\eta| \right) \\
& \nonumber\ge |w-y|\,|\eta|\,2^{-s\gamma}\left(\frac{1}{2}s^2-\frac{1}{4}\right) \\
& \ge 2^{j-k_2-s\gamma-5}s^2.\label{6.9}
\end{align}
Therefore, integrating by parts in $t$ variable gives us  
\begin{align*}
&\Lambda_{k_1}(I-G_{s,\nu})\mathcal{K}_{j,y}^{s,\nu}[\Lambda_{k_2}(I-S_{s,\nu})a](x)\\
=&\int_{\mathbb{R}^d} K_j^{s,\nu}(w-y)
\int_{\mathbb{R}^d}\int_{\mathbb{R}^d}
\hat{a}(\eta)\,\tilde{b}_{k_2,s,\nu}(\eta)\,\frac{\left(I-\frac{2^{-2k_1}\Delta_\xi}{4\pi^2}\right)^{N_1}
\big[b_{k_1,s,\nu}(\xi)\big]}{\big(1+2^{-2k_1}|x-w|^2\big)^{N_1}}
\\
&\qquad \times
e^{2\pi i \langle x-w,\xi\rangle}
\left[\big(\frac{i}{2\pi}\big)^{N_3}
\frac{1}{\langle w-y,\eta\rangle^{N_3}}
\int \vartheta_s^{(N_3)}(t)\,
e^{2\pi i \langle w-y,\eta\rangle t}\,dt
\right]
\,d\xi\,d\eta\,dw.
\end{align*}
Using \eqref{4.3}, \eqref{6.1}, \eqref{6.8} and \eqref{6.9} with $N_1>\frac{d}{2}$, we get that 
\begin{align*}
&\big\|\Lambda_{k_1}(I-G_{s,\nu})\mathcal{K}_{j,y}^{s,\nu}[\Lambda_{k_2}(I-S_{s,\nu})a]\big\|_{L^1(\mathbb{R}^d)}\\
&\lesssim_{N_1,N_3}
\left(2^{j-k_2-s\gamma-5}\,s^{2}\right)^{-N_3}
\|\vartheta_s^{(N_3)}\|_{L^1(\mathbb{R})}  \left(2^{s\gamma}s^{-14}\right)^{2N_1}
\,2^{-k_2 d}\,\|\hat{a}\|_{L^\infty(\mathbb{R}^d)}
\,\|K_j^{s,\nu}\|_{L^1(\mathbb{R}^d)}\\
&\lesssim_{N}
2^{-s\gamma(d-1)}\,
s^{11(N+d)}\,
2^{s\gamma(d+3)}\,
2^{(k_1-j+s\gamma)N},
\end{align*}
where $N_1 \in \left\{\frac{d+1}{2},\,\frac{d+2}{2}\right\}$ and $N:=N_3-d$.\\
\end{proof}

With these estimates in hand, we now complete the proof of Lemma~\ref{lem 3.3}. 

For fixed $ s, \nu $\and $j$, and for a fixed cube $Q,$ with $\ell(Q)=2^{j-s}$ we define
\begin{align*}
I_{k_1}
&:= \tilde\Lambda_{k_1}(I-P_{j-s+\ell(s)})\,\Lambda_{k_1}(I-G_{s,\nu})
\left[\int b_Q(y)\,\mathcal{K}_{j,y}^{s,\nu}[V_{j-s+\ell(s)}a]\,dy\right], \\
II_{k_1,k_2}
&:= \tilde\Lambda_{k_1}(I-P_{j-s+\ell(s)})\,\Lambda_{k_1}(I-G_{s,\nu})
\left[\int b_Q(y)\,\mathcal{K}_{j,y}^{s,\nu}[\Lambda_{k_2}a]\,dy\right].
\end{align*}
By \eqref{6.4} and \eqref{6.5}, it is enough to show that 
\begin{align*}
\sum_{k_1} \|I_{k_1}\|_{L^1(\mathbb{R}^d)}
+ \sum_{k_1}\sum_{k_2<j-s+\ell(s)}
\|II_{k_1,k_2}\|_{L^1(\mathbb{R}^d)}
\;\lesssim\;
s^{-11}\,2^{-s\gamma(d-1)}\,\|b_Q\|_{L^1(\mathbb{R}^d)}.\label{6.10}
\end{align*}
Since 
\[
\tilde\Lambda_{k_1}(I-P_{j-s+\ell(s)})g
= (\check{\tilde\beta}_{k_1}
- \check{\tilde\beta}_{k_1}*\Phi_{j-s+\ell(s)})*g, 
\] 
by the cancellation property of the kernel, we have that 
\begin{align}\label{6.11}
\|\tilde\Lambda_{k_1}(I-P_{j-s+\ell(s)})\|_{L^1(\mathbb{R}^d)\to L^1(\mathbb{R}^d)}
\lesssim \min\left\{1,\;2^{\,j-s+\ell(s)-k_1}\right\}.
\end{align}
By Lemma \ref{lem6.1}, \eqref{6.11} and \eqref{6.6}, we get that 
\begin{align*}
\sum_{k_1 \ge j-s+\ell(s)-10}
\|I_{k_1}\|_{L^1(\mathbb{R}^d)}
&\lesssim \sum_{k_1 \ge j-s+\ell(s)-10}
2^{ j - s + \ell(s)-k_1}\,s^{-11}\,2^{-s\gamma(d-1)}\,
\|b_Q\|_{L^1(\mathbb{R}^d)} \notag\\
&\lesssim s^{-11}\,2^{-s\gamma(d-1)}\,
\|b_Q\|_{L^1(\mathbb{R}^d)}.%\label{6.12}
\end{align*}
Note that 
\begin{align*}
&\Lambda_{k_1}(I-G_{s,\nu})
\left[\int b_Q(y)\,\mathcal{K}_{j,y}^{s,\nu}[V_{j-s+\ell(s)}a]\,dy\right](x)\\
&= \widetilde{\Lambda}_{k_1}(I-G_{s,\nu})\,\Lambda_{k_1}\!\left[\int b_Q(y)\,\mathcal{K}_{j,y}^{s,\nu}[V_{j-s+\ell(s)}a]\,dy\right](x),
\end{align*}
and by Fubini's theorem
\begin{align*}
& \Lambda_{k_1}\!\left[\int b_Q(y)\,\mathcal{K}_{j,y}^{s,\nu}[V_{j-s+\ell(s)}a]\,dy\right](x)\\
&= \int_{\mathbb{R}^d}
(\varphi -\varphi(2\cdot))\check{}_{k_1}(x-z)
\left( \int b_Q(y)\,\mathcal{K}_{j,y}^{s,\nu}[V_{j-s+\ell(s)}a]\,dy \right)(z)\,dz \\
&= \int
\Lambda_{k_1}\mathcal{K}_{j,y}^{s,\nu}[V_{j-s+\ell(s)}a](x)\,
b_Q(y)\,dy.
\end{align*}
Therefore, we have 
\begin{align*}
&\left\|\Lambda_{k_1}(I-G_{s,\nu})
\int b_Q(y)\,\mathcal{K}_{j,y}^{s,\nu}[V_{j-s+\ell(s)}a]\,dy
\right\|_{L^1(\mathbb{R}^d)} \\
&= \left\|\widetilde{\Lambda}_{k_1}(I-G_{s,\nu})\Lambda_{k_1}
\left[\int b_Q(y)\,\mathcal{K}_{j,y}^{s,\nu}[V_{j-s+\ell(s)}a]\,dy\right]
\right\|_{L^1(\mathbb{R}^d)} \\
&\le \|\widetilde{\Lambda}_{k_1}(I-G_{s,\nu})\|_{L^1(\mathbb{R}^d)\to {L^1(\mathbb{R}^d)}}\,
\left\|\Lambda_{k_1}
\left[\int b_Q(y)\,\mathcal{K}_{j,y}^{s,\nu}[V_{j-s+\ell(s)}a]\,dy\right]
\right\|_{L^1(\mathbb{R}^d)} \\
&\le \|\widetilde{\Lambda}_{k_1}(I-G_{s,\nu})\|_{{L^1(\mathbb{R}^d)}\to {L^1(\mathbb{R}^d)}}\,
\|b_Q\|_{L^1(\mathbb{R}^d)}\,
\sup_{y \in Q}\left\|\Lambda_{k_1}\mathcal{K}_{j,y}^{s,\nu}[V_{j-s+\ell(s)}a]
\right\|_{L^1(\mathbb{R}^d)}.
\end{align*}
Thus by \eqref{6.6}, the second estimate of Lemma \ref{lem6.2} $(ii)$ with $k_2=j-s+\ell(s)$, we get that 
\begin{align}
\sum_{k_1 < j-s+\ell(s)-10}
\|I_{k_1}\|_{L^1(\mathbb{R}^d)}
&\lesssim_N 2^{-s\gamma(d-2)}\sum_{k_1 < j-s+\ell(s)-10}
2^{(k_1 - j + s - \ell(s))d}\,2^{(k_1-j+s\gamma)N}\,
\|b_Q\|_{L^1(\mathbb{R}^d)} \notag\\
&\lesssim_N 2^{-s\gamma(d-2)}\,s^{11N}\,2^{s(\gamma-1)N}\,
\|b_Q\|_{L^1(\mathbb{R}^d)}.\label{6.13}
\end{align}
Further, by \eqref{6.11} and Lemma \ref{lem6.2} $(i)$, we have that 
\begin{align}
&\sum_{\substack{k_1 \ge j-s+\ell(s)-10\\k_2 < \min\{k_1-\ell(s)-10,\; j-s+\ell(s)\}}}
\|II_{k_1,k_2}\|_{L^1(\mathbb{R}^d)} \notag\\
&\lesssim_N \sum_{\substack{k_1 \ge j-s+\ell(s)-10\\k_2 < \min\{k_1-\ell(s)-10,\; j-s+\ell(s)\}}}
2^{\,j-s+\ell(s)-k_1}\,
2^{-s\gamma(d-2)}\,s^{11d+11N}\,2^{(k_2-j+s\gamma)N}
\|b_Q\|_{L^1(\mathbb{R}^d)} \notag\\
&\lesssim_N 2^{-s\gamma(d-2)}\,s^{11d+22N}\,2^{s(\gamma-1)N}\|b_Q\|_{L^1(\mathbb{R}^d)},\label{6.14}
\end{align}
Now apply Lemma \ref{lem 6.3} to deduce that 
\begin{align}
& \sum_{\substack{k_1 \ge j-s+\ell(s)-10\\k_1-\ell(s)-10 \le k_2 \le k_1+10\\k_2<j-s+\ell(s)}}
\|II_{k_1,k_2}\|_{L^1(\mathbb{R}^d)} \notag\\
&\lesssim_N \ell(s)\,
2^{4s\gamma}\,s^{\frac{11(N+d)}{\varepsilon}}\,\sum_{k_1 \leq j-s+2\ell(s)+10}
2^{(k_1-j+s\gamma)N}\,\|b_Q\|_{L^1(\mathbb{R}^d)}\notag\\
&\lesssim_N 2^{4s\gamma}s^{11(N+d)(\varepsilon^{-1}+2)}\,2^{s(\gamma-1)N}\;\|b_Q\|_{L^1(\mathbb{R}^d)}.\label{6.15}
\end{align}
%When $k_1 \ge j - s + \ell(s) - 10$ , $k_2 \not> k_1 + 10$, since $k_2 < j - s + \ell(s)$. Now sum the terms $\|II_{k_1,k_2}\|_{L^1(\mathbb{R}^d)}$ with $k_1 < j - s + \ell(s) - 10$.\\
Note that by Lemma~\ref{lem6.2} $(i)$ and estimates \eqref{6.6} and \eqref{6.11}, we have 
\begin{align}
&\sum_{\substack{k_1 < j-s+\ell(s)-10 \\ k_2 < k_1-\ell(s)-10}}
\|II_{k_1,k_2}\|_{L^1(\mathbb{R}^d)} \notag\\
&\lesssim_N 2^{-s\gamma(d-2)}\,s^{11d+11N}  \sum_{k_1 < j-s+\ell(s)-10}
\sum_{k_2 < k_1-\ell(s)-10}
2^{(k_2-j+s\gamma)N}\,\|b_Q\|_{L^1(\mathbb{R}^d)} \notag\\
&\lesssim_N 2^{-s\gamma(d-2)}\,s^{11d+11N}  2^{s(\gamma-1)N}\,\|b_Q\|_{L^1(\mathbb{R}^d)}.\label{6.16}
\end{align}
Similarly, by Lemma \ref{lem6.2} $(ii)$ and estimates \eqref{6.6} and \eqref{6.11}, we get that 
\begin{align}
&\sum_{\substack{k_1 < j-s+\ell(s)-10 \\ k_1+10<k_2 < j-s+\ell(s)}}\|II_{k_1,k_2}\|_{L^1(\mathbb{R}^d)}\notag \\
&\lesssim_N2^{-s\gamma(d-2)}
\sum_{k_1 < j-s+\ell(s)-10}
2^{(k_1-j+s\gamma)N}
\sum_{k_2 > k_1+10} 2^{(k_1-k_2)d}\,
\|b_Q\|_{L^1(\mathbb{R}^d)}\notag \\
&\lesssim 2^{-s\gamma(d-2)}\,s^{11N}\,2^{s(\gamma-1)N}\,
\|b_Q\|_{L^1(\mathbb{R}^d)}.\label{6.17}
\end{align}

Finally, Lemma \ref{lem 6.3} yields 
\begin{align}
&\sum_{\substack{k_1 < j-s+\ell(s)-10 \\ k_1-\ell(s)-10 \le k_2 \le k_1+10}}
\|II_{k_1,k_2}\|_{L^1(\mathbb{R}^d)} \notag\\
&\lesssim_N \ell(s)\,s^\frac{11(N+d)}{\varepsilon}\,2^{4s\gamma}
\sum_{k_1 < j-s+\ell(s)-10}
2^{(k_1-j+s\gamma)N}\,\|b_Q\|_{L^1(\mathbb{R}^d)} \notag\\
&\lesssim_N 2^{4s\gamma}\,s^{11(N+d)(\varepsilon^{-1}+1)}\,2^{s(\gamma-1)N}\,\|b_Q\|_{L^1(\mathbb{R}^d)}.\label{6.18}
\end{align}
Putting the estimates \labelcref{6.13,6.14,6.15,6.16,6.17,6.18} together, we finally have  
\begin{align*}
&\sum_{k_1 < j-s+\ell(s)-10}
\|I_{k_1}\|_{L^1(\mathbb{R}^d)}
+ \sum_{k_1}\sum_{k_2 < j-s+\ell(s)}
\|II_{k_1,k_2}\|_{L^1(\mathbb{R}^d)} \\
&\lesssim_N\,2^{A_1s}s^{A_2N}\,
2^{s(\gamma-1)N}\,
\|b_Q\|_{L^1(\mathbb{R}^d)}\\
&\lesssim s^{-11}\,2^{-s\gamma(d-1)}\,
\|b_Q\|_{L^1(\mathbb{R}^d)},
\end{align*}
where $s>s(\varepsilon), A_1:=4, A_2:=44\varepsilon^{-1}d$ and $N:=100d$.

Observe that in the estimates \labelcref{6.13,6.14,6.15,6.16,6.17,6.18}, the highest power of $2^s$ is $4\gamma$ which is smaller than $4$. Also, the highest power of $s$ is $11(N+d)(\varepsilon^{-1}+2)$ which is smaller than $44N\varepsilon^{-1}d$ for $N>d$. Finally, noting that $s^{-1}\log s \le 10^{-4}\frac{\varepsilon}{d}$ for $s>s(\varepsilon)$, we complete the proof of the lemma.
\qed
\section{Proof of Lemma \ref{lem 3.4}}
Note that, for fixed $s>0$, we have \begin{align}\label{7.1} \sup_{\xi \neq 0}\sum_{\nu \in \Theta_s} \left|\phi\!\left(2^{s\gamma}s^{-14}\langle \nu,\frac{\xi}{|\xi|}\rangle\right)\right| \lesssim 2^{s\gamma(d-2)}\,s^{14}.\end{align} 
We first prove the above claim.
Observe that it is sufficient to take the supremum over $\xi \in \mathbb S^{d-1}$. For $|\xi|=1$ and for $\phi\!\left(2^{s\gamma}s^{-14}\langle \nu,\xi\rangle\right)\neq 0$ we have $|\langle \nu,\xi\rangle| \leq 2^{-s\gamma}\,s^{14}$. If $|\langle \nu,\xi\rangle| \le 2^{-s\gamma}\,s^{14}$ and $\theta \in \mathbb S^{d-1}$ with $|\theta-\nu|\le 2^{-5-s\gamma}$, then 
$$|\langle \theta,\xi\rangle| \le |\langle \theta-\nu,\xi\rangle| + |\langle \nu,\xi\rangle| \leq 2^{-5-s\gamma} + 2^{-s\gamma}\,s^{14}.$$
Thus, if $|\langle \nu,\xi\rangle| \le 2^{-s\gamma}s^{14}$, we have 
\[
E_s^\nu:=\{\theta \in \mathbb S^{d-1} : |\theta-\nu| < 2^{-5-s\gamma}\}
\subset \{\theta \in \mathbb S^{d-1} : |\langle \theta,\xi\rangle| \leq 2^{-s\gamma+1}s^{14}\}.
\]
Let $m_s$ be the count of $\nu \in \Theta_s$ satisfying $|\langle \nu,\xi\rangle| \le 2^{-s\gamma}s^{14}$. Since $\{E^{\nu}_s\}_{\nu \in \Theta_s}$ are pairwise disjoint, we have that 
$$m_s \left(2^{-5-s\gamma}\right)^{d-1} \lesssim 2^{-s\gamma}s^{14}\Rightarrow m_s \lesssim 2^{s\gamma(d-2)}\,s^{14}.$$ 
 By applying Plancherel's theorem, Cauchy--Schwarz inequality and the estimate \eqref{7.1}, we get that 
\begin{align*}
&\quad\left\|\sum_{\nu \in \Theta_s}
G_{s,\nu}\!\left[\sum_{j\in\mathbb{Z}}
(I-P_{j-s+\ell(s)})T_j^{s,\nu}B^{\,j-s}\right]\right\|_{L^2(\mathbb{R}^d)}^2\\
&\le \int_{\mathbb{R}^d}
\left(\sum_{\nu \in \Theta_s}
\left|\phi\!\left(2^{s\gamma}s^{-14}\langle \nu,\frac{\xi}{|\xi|}\rangle\right)\right|^2
\right)
\left(\sum_{\nu \in \Theta_s}
\left|\big(\sum_{j\in\mathbb{Z}}
(I-P_{j-s+\ell(s)})T_j^{s,\nu}B^{\,j-s}\big)\hat{}(\xi)\right|^2
\right)
d\xi \\
&\lesssim 2^{s\gamma(d-2)}\,s^{14}
\int_{\mathbb{R}^d}
\sum_{\nu \in \Theta_s}
\left|\big(\sum_{j\in\mathbb{Z}}
(I-P_{j-s+\ell(s)})T_j^{s,\nu}B^{\,j-s}\big)\hat{}(\xi)
\right|^2 d\xi \\
&\lesssim 2^{s\gamma(d-2)}\,s^{14}
\sum_{\nu \in \Theta_s}
\left\|\sum_{j\in\mathbb{Z}}
(I-P_{j-s+\ell(s)})T_j^{s,\nu}B^{\,j-s}
\right\|_{L^2(\mathbb{R}^d)}^2.
\end{align*}
In view of the estimate $\operatorname{card}(\Theta_s)\lesssim 2^{s\gamma(d-1)},$ it suffices to show that for each $\nu \in \Theta_s$, the following holds 
\begin{align}\label{7.2}
\left\|\sum_{j\in\mathbb{Z}}
(I-P_{j-s+\ell(s)})T_j^{s,\nu}B^{\,j-s}
\right\|_{L^2(\mathbb{R}^d)}^2
\lesssim 2^{-2s\gamma(d-1)}\alpha
\sum_{j\in\mathbb{Z}} \|B^{\,j-s}\|_{L^1(\mathbb{R}^d)}.
\end{align}
Note that
\begin{align*}
\big|(I-P_{j-s+\ell(s)})T_j^{s,\nu}B^{\,j-s}(x)\big|
&\leq\left( |K_j^{s,\nu}| + \big|\Phi_{j-s+\ell(s)}\big| * |K_j^{s,\nu}| \right) * |B^{\,j-s}|(x).
\end{align*}
Denote $$H_j^{s,\nu}(x) := 2^{-jd}\,\chi_{\tau_j^{s,\nu}}(x),$$
where
\[
\tau_j^{s,\nu}
:= \left\{ x \in \mathbb{R}^d :
|\langle x,\nu\rangle| \le 2^{j+3},\;
|x - \langle x,\nu\rangle \nu| \le 2^{j+3-s\gamma}
\right\}.
\]
Therefore by the estimates \eqref{4.3}, \eqref{4.7} and using the properties of $\Phi$, we have
\[\big|(I-P_{j-s+\ell(s)})T_j^{s,\nu}B^{\,j-s}(x)\big|\lesssim \big|H^{s,\nu}_j*B^{j-s}(x)\big|.\]
Therefore, we apply the $TT^*$ argument for any $N \in \mathbb{N}$ to get that 
\begin{align*}
&\quad\left\|\sum_{j=-N}^{N}
(I-P_{j-s+\ell(s)})T_j^{s,\nu}B^{\,j-s}
\right\|_{L^2(\mathbb{R}^d)}^2\notag\\
&\lesssim \sum_{j=-N}^{N}
\int_{\mathbb{R}^d}
\big(H_j^{s,\nu} \ast H_j^{s,\nu} \ast |B^{\,j-s}|\big)(x)\,
|B^{\,j-s}|(x)\,dx\notag \\
&\quad + 2\sum_{j=-N}^{N}\sum_{i=-N}^{j-1}
\left(\Re \int_{\mathbb{R}^d}
\big(H_j^{s,\nu} \ast H_i^{s,\nu} \ast |B^{\,i-s}|\big)(x)\,
|B^{\,j-s}|(x)\,dx\right)\notag\\
&\lesssim 2\sum_{j\in\mathbb{Z}}
\int_{\mathbb{R}^d}
|B^{\,j-s}(x)|
\left( \sum_{i\le j}
H_j^{s,\nu} \ast H_i^{s,\nu} \ast |B^{\,i-s}|(x) \right)
dx.\label{7.3}
\end{align*}
Note that
\[
\|H_i^{s,\nu}\|_{L^1(\mathbb{R}^d)}
\leq 2^{-id}\,|\tau_i^{s,\nu}|
\lesssim 2^{-s\gamma(d-1)},
\]
and hence for $i\leq j$, we get that
\[
|H_j^{s,\nu} \ast H_i^{s,\nu}(x)|
\le 2^{-s\gamma(d-1)}\,2^{-jd}\,
\chi_{\widetilde{\tau}_j^{s,\nu}}(x),
\]
where $\widetilde{\tau}_j^{s,\nu}:= \left\{ x \in \mathbb{R}^d :
|\langle x,\nu\rangle| \le 2^{j+4},\;
|x - \langle x,\nu\rangle \nu| \le 2^{j+4-s\gamma}
\right\}.$

Thus for each $x \in \mathbb{R}^d$ and $j \in \mathbb{Z}$,
\begin{align*}
\sum_{i \le j} H_j^{s,\nu} \ast H_i^{\nu} \ast |B^{i-s}|(x)
&\le\sum_{i \le j}
2^{-s\gamma(d-1)}\,2^{-jd}
\int_{x+\widetilde{\tau}_j^{s,\nu}} |B^{i-s}(y)|\,dy \\
&\le \sum_{i \le j}
2^{-s\gamma(d-1)}\,2^{-jd}
\sum_{\substack{Q \in \mathcal{Q}:\\ \ell(Q)=2^{i-s} \\ Q \cap (x+\widetilde{\tau}_j^{s,\nu}) \neq \emptyset}}
\int_Q |b_Q(y)|\,dy \\
&\lesssim 2^{-s\gamma(d-1)}\,2^{-jd}\,\alpha\,\sum_{i \le j}\sum_{\substack{Q \in \mathcal{Q}:\\ \ell(Q)=2^{i-s} \\ Q \cap (x+\widetilde{\tau}_j^{s,\nu}) \neq \emptyset}}|Q|\\
&\lesssim 2^{-2s\gamma(d-1)}\,\alpha,
\end{align*}
Note that with this assertion, the estimate \eqref{7.2} follows and hence the proof is complete.
\qed
%%%
\section*{Acknowledgement}
Saurabh Shrivastava acknowledges financial support by Anusandhan National Research Foundation, India under the project ANRF/ARG/2025/000940/MS. Sangita Kundu is supported by the Indian Institute of Science Education and Research Bhopal for her PhD fellowship.
\bibliography{biblio}

\end{document}